\documentclass[a4paper,12pt]{amsart}
\usepackage[utf8]{inputenc}
\usepackage[T1]{fontenc}
\usepackage[UKenglish]{babel}
\usepackage[a4paper,margin=28mm]{geometry}
\usepackage{verbatim}

\allowdisplaybreaks[4]
\usepackage{times}
\usepackage{dsfont,mathrsfs}
\usepackage{amsmath}
\usepackage{amsthm}
\usepackage{amssymb}
\usepackage{amsfonts}
\usepackage{latexsym}
\usepackage{booktabs}

\usepackage[colorlinks,linkcolor=black,pagebackref]{hyperref}
\usepackage{xcolor}

\usepackage{ulem}

\newtheorem{theorem}{Theorem}[section]
\newtheorem{lemma}[theorem]{Lemma}
\newtheorem{proposition}[theorem]{Proposition}

\theoremstyle{definition}

\newtheorem{remark}[theorem]{Remark}
\numberwithin{equation}{section}
\renewcommand{\labelenumi}{\roman{enumi})}
\renewcommand\theenumi\labelenumi
\renewcommand{\leq}{\leqslant}
\renewcommand{\le}{\leqslant}
\renewcommand{\geq}{\geqslant}
\renewcommand{\ge}{\geqslant}

\newcommand{\tl}{\tilde}
\newcommand{\Be}{\begin{equation}}
\newcommand{\Ees}{\end{equation*}}
\newcommand{\Bes}{\begin{equation*}}
\newcommand{\Ee}{\end{equation}}

\newcommand{\R}{\mathbb{R}}
\newcommand{\E}{\mathbb{E}}
\newcommand{\e}{\varepsilon}
\newcommand{\PP}{\mathbb{P}}
\newcommand{\N}{\mathbb{N}}
\newcommand{\mcl}{\mathcal}

\newcommand{\dif}{\mathrm{d}}

\usepackage{marginnote}
\usepackage[notref,notcite]{showkeys}

\begin{document}

\title[An approximation to steady-state of M/Ph/n+M queue]
{An approximation to steady-state of M/Ph/n+M queue}

\author[X. H. Jin]{Xinghu Jin}
\address{Xinghu Jin: 1. Department of Mathematics,
Faculty of Science and Technology,
University of Macau,
Av. Padre Tom\'{a}s Pereira, Taipa
Macau, China; \ \ 2. UM Zhuhai Research Institute, Zhuhai, China.}
\email{yb77438@connect.um.edu.mo}

\author[G. Pang]{Guodong Pang}
\address{Guodong Pang: The Harold and Inge Marcus Department of Industrial and Manufacturing Engineering, College of Engineering, Pennsylvania State University, University Park, PA 16802}
\email{gup3@psu.edu}

\author[L. Xu]{Lihu Xu}
\address{Lihu Xu: 1. Department of Mathematics,
Faculty of Science and Technology,
University of Macau,
Av. Padre Tom\'{a}s Pereira, Taipa
Macau, China; \ \ 2. UM Zhuhai Research Institute, Zhuhai, China.}
\email{lihuxu@umac.mo}

\author[X. Xu]{Xin Xu}
\address{Xin Xu: 1. Department of Mathematics,
Faculty of Science and Technology,
University of Macau,
Av. Padre Tom\'{a}s Pereira, Taipa
Macau, China; \ \ 2. UM Zhuhai Research Institute, Zhuhai, China.}
\email{yb77439@umac.mo}


\begin{abstract}
In this paper, we develop a stochastic algorithm based on Euler-Maruyama scheme to approximate the invariant measure of the limiting multidimensional diffusion of the $M/Ph/n+M$ queue.
Specifically, we prove a non-asymptotic error bound between the invariant measures of the approximate model from the algorithm and the limiting diffusion of the queueing model.
Our result also provides an approximation to the steady-state of the diffusion-scaled queueing processes in the Halfin-Whitt regime given the well established interchange of limits property.
To establish the error bound, we employ the recently developed Stein's method for multi-dimensional diffusions, in which the regularity of Stein's equation developed by Gurvich  \cite{Gur1} plays a crucial role.

We further prove the central limit theorem (CLT) and the moderate deviation principle (MDP) for the occupation measures of the limiting diffusion of the $M/Ph/n+M$ queue and its Euler-Maruyama scheme. In particular, the variance of the CLT of the limiting queue is determined by using Stein's equation and Malliavin calculus.

\end{abstract}

\maketitle

\tableofcontents\thispagestyle{plain}

\section{Introduction} \label{Introduction}

A fundamentally important result in heavy-traffic queueing theory is the validity of diffusion approximations, that is, the interchange of limits property.
It provides an approximation of the steady state distribution of the queueing processes in a heavy-traffic regime by using the invariant measure of the limiting diffusion.
For instance, the interchange of limits results are proved for stochastic networks in \cite{GZ,BL,G14,YY16,YY18} and for many server queues \cite{DDG1,GS,S15,AHP1}.
For some queueing models, the invariant measures of the limiting diffusions can be explicitly characterized \cite{HW87, DDG1}.
When this is impossible, numerical schemes are often drawn upon to compute the invariant measures, for example, computation of invariant measures of Reflected Brownian motions in \cite{DH92,DH1,BCR1}. Our paper  is of similar flavor as \cite{BCR1} where a Euler scheme approximation is developed for the constrained diffusions arising as scaling limits of stochastic networks.

In this paper, we focus on the  $M/Ph/n+M$ model in the Halfin-Whitt regime.
For many-server queues with exponential services, the limiting diffusions of the scaled queueing processes are one-dimensional with a piecewise-linear drift, whose steady state distributions have explicit expression as shown in \cite{BW,DDG1}. However, for  many-server queues with phase-type service time distributions, the limiting diffusions are multidimensional with a piecewise-linear drift, as shown in \cite{PR,DHT1}. Although the validity of diffusion approximations is proved for the $M/Ph/n+M$ queues in \cite{DDG1},  the multi-dimensional limiting diffusion does not have an explicit invariant measure \cite{DG1}.
In fact, characterization of multi-dimensional piecewise diffusions has been left  as an open problem thus far in \cite{BW}. The objective in this paper is to provide an approximation for the invariant measure of the limiting diffusion of the  $M/Ph/n+M$ model, and thus also an approximation for the steady-state of the diffusion-scaled queueing processes in all phases for the model in the Halfin--Whitt regime.

\subsection{Summary of results and contributions}
The limiting diffusion  $(X_t)_{t\geq 0}$ satisfies the following stochastic differential equation (for short, SDE):
\begin{eqnarray}\label{hSDEg}
\dif X_{t} &=& g (X_{t} ) \dif t + \sigma \dif B_t
\end{eqnarray}
with $(B_t)_{t\geq 0}$ being a $d$-dimensional standard Brownian motion and
\begin{eqnarray*} 
g(x) &=& -\beta p-Rx+(R-\alpha I)p({\rm e}' x)^+, \ \ \forall x\in \R^d\,.
\end{eqnarray*}
Here  $y^+=\max\{0,y\}$ for all $y\in \R$,   $\alpha> 0$ is the patience rate,  $\beta$ is the slack in the arrival rate relative to a critically loaded system,  $p \in \R^d$ is a vector of non-negative entries whose sum is equal to one, ${\rm e}=(1,1,\cdots,1)'$ with $'$ denoting the transpose,  $I$ is the identity matrix,
 \begin{eqnarray*}
R \ = \ (I-P') \textrm{diag}(v), \quad \frac{1}{\zeta} \ = \ {\rm e}' R^{-1}p, \quad
\gamma \ = \ \zeta R^{-1}p,
\end{eqnarray*}
where $v=(v_1,\cdots,v_d)$ with $v_k$ being the service rate in phase $k$,
 and $P$ be a sub-stochastic matrix describing the transitions between service phases such that $P_{ii}=0$ for $i=1,\cdots,d$, and $I-P$ being invertible (\cite[Section 2.2]{DHT1}).
 Assume that the Ph phase distribution has mean $1$, that is, $\zeta=1$. It is easy to check ${\rm e}' \gamma=1$.
 $\sigma \sigma'$ has the following form:
\begin{eqnarray*}
\sigma \sigma' &=& \textrm{diag} (p) + \sum_{k=1}^d \gamma_k v_k H^{(k)} + (I- P')\textrm{diag}(v) \textrm{diag}(\gamma)(I-P),
\end{eqnarray*}
where $\gamma=(\gamma_1,\cdots,\gamma_d)$, $H^{(k)}=(H^{(k)}_{ij})_{1\leq i, j \leq d}\in \R^{d\times d}$ with $H_{ii}^{(k)} = P_{ki}(1-P_{ki})$ and $H_{ij}^{(k)}=-P_{ki}P_{kj}$ for $j\neq i$.  Throughout this paper, we assume that there exists some constant $c>0$ such that  $\xi^{'}\sigma\sigma^{'}\xi \geq c\xi^{'}\xi$ for all  $\xi \in \R^d$. It is shown in \cite[Theorem 3]{DG1} that  the diffusion  $(X_t)_{t\geq 0}$ is exponentially ergodic and admits a unique invariant measure $\mu$. It is well known that for every $x\in \R^d$, SDE (\ref{hSDEg}) has a unique solution $(X^{x}_t)_{t\ge 0}$ starting from $x$ and the solution is nonexplosive from \eqref{e:AV} below  and \cite[Theorem 2.1]{MT2}.

The {\bf EM scheme} that we design to approximate the invariant measure $\mu$ of $(X_t)_{t\geq 0}$ reads as the following:
\begin{eqnarray}\label{e:XD}
\tl{X}_{k+1}^{\eta}
&=& \tl{X}_{k}^{\eta}+g(\tl{X}_{k}^{\eta}) \eta + \sqrt{\eta} \sigma \xi_{k+1},
\end{eqnarray}
where $k\in \mathbb{N}_0\triangleq \mathbb{N}\cup\{0\}$, $\tl{X}^{\eta}_{0}$ is the initial value and $\{\xi_k\}_{k\in \mathbb{N}}$ are the independent standard $d$-dimensional Gaussian random variables.

Our first main result is the following theorem about this EM scheme, which provides a non-asymptotic estimate for the error between the ergodic measures of the SDE and its EM scheme.

\begin{theorem}\label{thm:DDE}
$(\tl{X}_{k}^{\eta})_{k\in \mathbb{N}_0}$ defined by \eqref{e:XD} admits a unique invariant measure $\tl \mu_{\eta}$ and is exponentially ergodic. Moreover, the following two statements hold:

(i) There exists some positive constant $C$, not depending on $\eta$, such that
\begin{eqnarray*}
d_{W}(\mu,\tl \mu_{\eta}) \ \leq \ C\eta^{\frac{1}{2}},
\end{eqnarray*}
where $d_W$ is the Wasserstein-1 distance, see \eqref{e:dW} for the definition.

(ii)  For any (small) $\delta>0$ and taking $\eta=\delta^2$, we can run the EM algorithm $N:=O(\delta^{-2} \log \delta^{-1})$ steps so that the law of $\tl{X}^\eta_N$, denoted as $\mcl L(\tl{X}^\eta_N)$, satisfies
$$d_W(\mcl L(\tl{X}^\eta_N),\mu) \ \leq \ \delta.$$
\end{theorem}

Our second set of main results are the central limit theorem (for short, CLT), the moderate deviation principle (for short, MDP) for the long term behavior of $(X_t)_{t\geq 0}$ and $(\tl{X}^{\eta}_k)_{k\in \mathbb{N}_0}$. For any $x \in \R^d$ and $T>0$, the empirical measure $\mcl E_T^x$ of $(X^x_t)_{t \ge 0}$ is defined by
 \begin{eqnarray*}
 \mcl E_T^x(A) \ = \ \frac 1T \int_0^T \delta_{X^x_s}(A) \dif s, \ \ \ \ \  A \in \mathcal{B}(\R^d),
 \end{eqnarray*}
where $\mathcal{B}(\R^d)$ is the collection of Borel sets on $\R^{d}$, $\delta_y(\cdot)$ is a delta measure, that is, $\delta_y(A)=1$ if $y \in A$ and $\delta_y(A)=0$ if $y \notin A$. It is easy to check that for any measurable function $h: \R^d \rightarrow \R$,
$$\mcl E_T^x(h) \ = \ \frac 1T \int_0^T h(X^x_s) \dif s.$$
For any $x \in \R^d$ and $n \in \mathbb{N}$, the empirical measure $ \mcl E_n^{\eta,x} $ of $(\tl{X}^{\eta,x}_k)_{k\in \mathbb{N}_0}$ is defined by
 \begin{eqnarray*}
 \mcl E_n^{\eta,x}(A) \ = \ \frac{1}{n} \sum_{k=1}^n \delta_{\tl{X}^{\eta,x}_k}(A), \ \ \ \ \  A \in \mathcal{B}(\R^d).
 \end{eqnarray*}
It is easy to check that for any measurable function $h: \R^d \rightarrow \R$,
$$\mcl E_n^{\eta,x}(h) \ = \ \frac{1}{n} \sum_{k=1}^n h(\tl{X}^{\eta,x}_k).$$

\begin{theorem} [CLT] \label{thm:CLT}
For any $h\in \mathcal{B}_b(\R^d,\R)$ and $x \in \R^d$,  $\sqrt{t} \left[\mcl E^x_t(h)-\mu(h)\right]$ weakly converges to  $\mathcal{N}(0,\mu(|\sigma^{\prime}  \nabla f|^2))$ as  $t \rightarrow \infty$, where $f$ is the solution to the Stein's equation \eqref{e:SE}. Furthermore, $\mu(|\sigma' \nabla f|^2)\leq C\|h\|^2_{\infty}<\infty$ for some $C>0$.
\end{theorem}

\begin{theorem}[MDP]\label{thm:MDP}
For any $h\in \mathcal{B}_b(\R^d,\R)$, $x \in \R^d$ and measurable set $A\subset \R$, one has
	\begin{eqnarray*}
	-\inf_{z \in A^{ {\rm o} } } \frac{ z^2 }{ 2 \mu(| \sigma^{\prime}  \nabla f |^2)  }
	\ &\leq& \ 
	\liminf_{t\to\infty}\frac{1}{a_t^2}\log \mathbb{P} \left( \frac{\sqrt{t}}{a_t}  \left[\mcl E^x_t(h)-\mu(h)\right]    \in A \right)  \\
	\ &\leq& \
	\limsup_{t\to\infty}\frac{1}{a_t^2}\log \mathbb{P} \left( \frac{\sqrt{t}}{a_t}  \left[\mcl E^x_t(h)-\mu(h)\right]  \in A \right)
\ \leq \ -\inf_{z \in \bar{A}} \frac{ z^2 }{ 2 \mu(| \sigma^{\prime}  \nabla f |^2)  },
	\end{eqnarray*}
	where $\bar{A}$ and $A^{ {\rm o} }$ are the closure and  interior of set $A$, respectively, and $a_t$ satisfies $a_t \to \infty$ and $\frac{a_t}{\sqrt{t}} \to 0$ as $t\to \infty$ and $f$ is the solution to the Stein's equation \eqref{e:SE}.
\end{theorem}

\begin{theorem}[CLT]\label{thm:EMCLT}
For any $h$ satisfying $|h|\leq V^{\ell}$ with some integer $\ell$ and $V$ be in \eqref{e:Lypfun} and for any initial distribution, 
$\sqrt{n} \big[ \mcl E_n^{\eta, \tl{X}_0^{\eta}}(h)  -  \tl{\mu}_{\eta}(h) \big]$ weakly converges to $\mathcal{N}(0, \sigma_h^2)$ as $n\to \infty$,
where 
     \begin{eqnarray}\label{e:sigmaf2}
\sigma_h^2  \ = \ {\rm var}_{\tl{\mu}_{\eta}} [ h(\tl{X}_0^{\eta}) ] + 2 \sum_{i=1}^{\infty} {\rm cov}_{\tl{\mu}_{\eta}} [ h(\tl{X}_0^{\eta}), h(\tl{X}_i^{\eta})].
	\end{eqnarray} 
\end{theorem}

\begin{theorem}[MDP]\label{thm:EMMDP}
For any $h\in \mathcal{B}_b(\R^d,\R)$, $x \in \R^d$ and measurable set $A\subset \R$, one has
	\begin{eqnarray*}
	-\inf_{z \in A^{ {\rm o} } } \frac{ z^2 }{2 \mathcal{V}(h)}
	\ &\leq& \ 
	\liminf_{n \to\infty}\frac{1}{a_n^2}\log \mathbb{P} \left( \frac{\sqrt{n}}{a_n}  \left[ \mcl E_n^{\eta,x}(h) - \tl{\mu}_{\eta}(h)\right]    \in A \right)  \\
	\ &\leq& \
	\limsup_{n \to\infty}\frac{1}{a_n^2}\log \mathbb{P} \left( \frac{\sqrt{n}}{a_n}  \left[ \mcl E_n^{\eta,x}(h) - \tl{\mu}_{\eta}(h)\right]  \in A \right)
\ \leq \ -\inf_{z \in \bar{A}} \frac{ z^2 }{ 2 \mathcal{V}(h) },
	\end{eqnarray*}
	where $\bar{A}$ and $A^{ {\rm o} }$ are the closure and  interior of set $A$, respectively,  and $a_n$ satisfies $a_n \to \infty$ and $\frac{a_n}{\sqrt{n}} \to 0$ as $n\to \infty$ and 
	     \begin{eqnarray}\label{e:EMVh}
\mathcal{V}(h) \ &=& \  \langle  (  h-\tl{\mu}_{\eta}(h)  )^2  \rangle_{\tl{\mu}_{\eta}} + 2\sum_{k=1}^{\infty} \langle \tl{\mathcal{P}}^k_{\eta} h, h-\tl{\mu}_{\eta}(h) \rangle_{\tl{\mu}_{\eta}}.
	\end{eqnarray}
\end{theorem}
\begin{remark} 
We shall see below that Stein's equation will play an important role in proving Theorems \ref{thm:DDE}, \ref{thm:CLT} and \ref{thm:MDP}.  As $\eta \rightarrow 0$ and $n \eta \rightarrow \infty$, we expect that the CLT in Theorem \ref{thm:EMCLT} holds  with the variance $\sigma^2_h$ replaced with $\mu(| \sigma^{\prime}  \nabla f |^2)$, however, the issue of exchanging limits of $\eta$ and $n$ seems very hard. When $h \in \mathcal{C}^2_b(\R^d, \R)$ and $g$ is second order differentiable and strongly dissipative, \cite{lu2020central} gives a proof of this conjecture.
\end{remark}

\subsection{Related works}

Our paper is relevant to the following four streams of works in the literature.

{\it (a) Steady state analysis of many-server queues.}
A few significant results have been obtained for understanding the steady-state of many-server queues, see, e.g., \cite{APS1,AHP1,AHPS1,GG13a,GG13b,AR20,Gur1,BD1,BDF} and references therein.
The most relevant to us is the recent development using Stein's method to analyze the steady state of queueing processes via diffusion approximations.
Gurvich \cite{Gur1} provides a framework of analyzing the steady states via direct diffusion approximations (rather than diffusion limits) for a family of continuous-time exponentially ergodic Markov processes with state spaces,
in particular, the gap between the steady-state moments of the diffusion models and those of the Markov processes is characterized.  This result can be applied to Markovian many-server queueing systems.
Braverman et al.  \cite{BDF} introduced the Stein's method framework formally, proving Wasserstein and Kolmogorov distances between the steady-state distributions of the queueing processes and approximate diffusion models, and applied to the classical Erlang A and C models, where the bound is characterized by the system size.
Braverman  and Dai \cite{BD1} then extended this approach for the $M/Ph/n+M$ queues.
 Braverman et al. \cite{BDF20}  recently studied high order steady-state approximations of 1-dimensional Markov chains and applied to Erlang C models.
As discussed before, invariant measure of the diffusion limit of the $M/Ph/n+M$ queues lacks an explicit expression and is difficult to compute directly. In this work, we provide a stochastic algorithm to compute the invariant measures for an approximate diffusion model of the $M/Ph/n+M$ queue.
Unlike the work \cite{BD1}, our work characterizes the non-asymptotic error bound in terms of the step size in the algorithm.

It is worth noting that for models with the one-dimensional diffusion approximations, since the invariant measure has an explicit density, when studying steady state approximation problem by Stein's method, one can explicitly solve Stein's equation and obtain the desired regularity properties easily. That is very similar to the one dimensional normal approximation case. It is well known that the generalization of Stein's method from one to multi-dimensional approximations is highly nontrivial \cite{CM08,RR-09}.
 Our problem is a multidimensional diffusion approximation.

{\it (b) Error estimates of EM schemes for diffusions.}
Let us recall the results concerning the error estimates between the ergodic measures of SDEs and their EM scheme.
For the ease of stating and comparing the results in the literatures below, we denote in this subsection by  $(Y_{t})_{t \ge 0}$ and $(\tl{Y}_{n})_{n \in \N_{0}}$ the stochastic processes associated to SDEs and their EM scheme respectively, and by $\pi_{sde}$ and $\pi_{em}$ their ergodic measures respectively.

There have been many results concerning the error estimates between the ergodic measures of $(Y_{t})_{t \ge 0}$ and $(\tl{Y}_{n})_{n \in \N_{0}}$, see for instance \cite{BDMS1, BCR1, BBC1, DM1, DM2, KT1, LP1, P1, P2, T3}, but most of them are
asymptotic type.  For asymptotic results, we recall those in the literatures \cite{GHL1, P1, P2,BCR1} whose settings are close to ours. \cite{GHL1, P1, P2} considered an empirical measure $\Pi^{n}_{em}$ of $\pi_{em}$ for a class of SDEs driven by multiplicative L\'evy noises,  showing that $\frac1{\sqrt{\Gamma_n}}(\Pi^{n}_{em}(f)-\pi_{sde}(f))$ converges to a normal distribution as $n \rightarrow \infty$ for $f$ in a certain high order differentiable function family ($\Gamma_n$ has the same order as $n$), while a similar type CLT was obtained in \cite{BCR1} for a reflected SDE driven arising in queue systems. All these works need strong dissipation and high order differentiability conditions on the drift of SDEs, which do not hold in our queue systems.

Among the few non-asymptotic results, the works in \cite{BDMS1,DM1,DM2}, arising from the Langevin dynamics sampling, are probably most close to ours.
Due to the distribution sampling motivation, their SDEs are gradient systems, i.e., the drift is the gradient of a potential $U$, thus analysis tools such as concentration inequalities are available.
Under certain conditions on the drift, they proved non-asymptotic bounds for total variation or Wasserstein-2 distance between $\pi_{sde}$ and $\pi_{em}$, their analysis heavily depends on the gradient form of the drift, and is not easily seen to be extended to a non-gradient system. Our SDE is not a gradient system in that its drift $g(x)$ can not be represented as a gradient of a potential, what is worse is that $g(x)$ is even not differentiable.

\cite{BBC1,T3} gave non-asymptotic results for the difference between the law of $\tl{Y}_n$ and $\pi_{sde}$ for large $n$. Most of these works need strong dissipation and high order differentiability assumptions on the drift of SDE, and their estimates are in the form
$|\E h(\tl{Y}_n)-\pi_{sde}(h)|$ or $|\frac{1}n \sum_{i=1}^n h(\tl{Y}_i)-\pi_{sde}(h)|$ for $h$ in a certain high order differentiable function family, from which one usually cannot derive a bound between the law of $\tl{Y}_n$ and $\pi_{sde}$ in a Wasserstein type distance.


{\it (c) CLT and MDP with respect to ergodic measures.}  \cite{DG1} proved that SDE (\ref{hSDEg}) is exponentially ergodic with ergodic measure $\mu$. 
Inspired by the work \cite{WLM1, WXX1}, we will establish the CLT and MDP with respect to $\mu$, in which the related variance can be determined by solving Stein's equation. Because the test functions $h$ in Eq. \eqref{e:SE} are in $\mathcal B_b(\R^d,\R)$, we need to apply Malliavin calculus to study the regularity of the solution. 
We also prove the CLT and MDP of the EM scheme, there exist very few results for studying CLT and MDP of EM scheme, see \cite{fukasawa2020efficient,lu2020central}.

\subsection{Organization of the paper}

In the remainder of this section, we introduce notations which will be frequently used. Sections \ref{sec-proof strategy}, \ref{s:CLTMDP} and \ref{sec:EMCLTMDP} give the proofs of our main results, in which several auxiliary propositions will be proved in Appendix \ref{App:GeneralErgodicEM}. In Section \ref{Malliavin-Stein},
we introduce two Stein's equations with different test functions $h$, and use the Malliavin calculus to get the regularity for solution when $h\in \mathcal{B}_b(\R^d,\R)$. 
The moment estimate for a weighted occupation time are in Section \ref{app:occupation}. Some proofs for auxiliary lemmas are in Appendix \ref{sec:AAS}.

\subsection{Notations}
Let $\R$ and $\mathbb{C}$ be real numbers and complex numbers respectively. The Euclidean metric is denoted by $| \cdot |$. For matrixes $A=(A_{ij})_{d\times d}$ and $B=(B_{ij})_{d\times d}$, denote $\langle A, B \rangle_{ {\rm HS} }=\sum_{i,j=1}^d A_{ij} B_{ij}$ and Hilbert Schmidt norm is $ \| A \|_{  {\rm HS} } = \sqrt{ \sum_{i,j=1}^d A_{ij}^2   } $ and operator norm is $\| A \|_{ {\rm op} } = \sup_{  |u|=1 } |Au|$.
We write a symmetric matrix $A>0\, (A<0)$ if $A$ is a positive (negative) definite matrix, and write $A\ge 0 \, (A \leq0)$ if $A$ is a positive (negative) semi-definite matrix. $\langle x, y \rangle$ means the inner product, that is, $\langle x, y \rangle = x^{\prime} y$ for $x,y\in \R^d$. $\otimes$ is the outer product, that is, for vector $u=(u_1,\cdots,u_d)$ and matrix  $A=(A_{ij})_{d\times d}$, then $(u \otimes A)_{ijk}=u_{i}A_{jk}$ for $1 \le i, j, k \le d$.

$\mathcal{C}^k(\R^d, \R_+)$ means $\R_+$-valued $k$-times continuous derivatives functions defined on $\R^d$ with $k\in \N$ and $\R_+=[0,\infty)$. $ \mathcal{C}_b(\R^d, \R)$ is $\R$-valued continuous bounded functions defined on $\R^d$. $\mathcal{B}_b(\R^d, \R )$ means $\R$-valued Borel bounded measurable functions defined on $\R^d$. Denote $\| f \|_{\infty}= {\rm ess}\sup_{x\in \R^d} |f(x)|$ for $f\in\mathcal{B}_b(\R^d, \R )$. For $f\in \mathcal{C}^2(\R^d, \R)$, denote $\nabla f=(\partial_1 f, \partial_2 f, \cdots, \partial_d f) \in \R^d$ and $\nabla^2 f= (   \partial_{ij} f)_{1\leq i,j \leq d} \in \R^{d\times d}$ the gradient and Hessian matrix for function $f$. For $f\in \mathcal{C}^1(\R^d, \R)$ and $u, x\in \R^d$, the directional derivative $\nabla_{u} f (x)$ is defined by
\begin{eqnarray*}
\nabla_{u} f(x) &=& \lim_{\e_1 \to 0} \frac{f(x+\e_1 u) - f(x)}{\e_1}.
\end{eqnarray*}
We know $\nabla f (x) \in \R^d$ for each $x \in \R^{d}$ and $\nabla_{u} f (x) =\langle \nabla f(x), u \rangle$. For $f \in \mathcal{C}^2(\R, \R)$, $\dot{f}$ and $\ddot{f}$ are the first and second derivatives of function $f$ respectively. For any probability measure $\nu$, denote $\nu(f) = \int f(x) \nu(\dif x)$.

$B(y,r)$ means the ball in $\R^d$ with centre $y \in \R^d$ and radius $r>0$, that is, $B(y,r)=\{z\in \R^d:|z-y|\leq r\}$.

$\mathcal{N}(a,A)$ with $a\in \R^d$ and $A\in \R^{d\times d}$ denotes Gaussian distribution with mean $a$ and covariance matrix $A$.

A sequence of random variables $\{Y_n, n\geq 1 \}$ is said to converge weakly or converge in distribution to a limit $Y_{\infty}$, that is, $Y_n \Rightarrow Y_{\infty}$ if $\lim_{n \rightarrow \infty} \E f(Y_{n})=\E f(Y_{\infty})$ for all bounded continuous function  $f$. In addition,  $Y_n \stackrel{p}{\longrightarrow} Y_{\infty}$ means convergence in probability, namely, $\lim_{n\to \infty} \PP(|Y_n - Y_{\infty}|>\delta)=0$ for all $\delta \geq 0$. $Y_n \stackrel{L^p}{\longrightarrow} Y_{\infty}$ means the $L^p$ convergence, that is, $\lim_{n\to \infty}\E |Y_n-Y_{\infty}|^p=0$.

Denote $X^x_t$ the process $X_t$ given $X_0=x$.  Denote by $P_{t}(x,\cdot)$ the transition probability of $X_{t}$ given $X_{0}=x$. Then the associated Markov semigroup $(P_t)_{t\ge 0}$ is given by,  for all $x\in\R^d$ and $f\in\mathcal{B}_b(\R^d,\R )$
\begin{eqnarray*}
P_tf(x) 
\ = \ \E f(X^x_t)
\  = \ \int_{\R^{d}} f(y) P_{t}(x, \dif y), \quad \forall t\ge0.
\end{eqnarray*}
The generator $\mathcal{A}$ of $(X_t)_{t\geq 0}$ is given by, for $y\in\R^d$,
\begin{eqnarray}\label{e:A}
\mathcal{A} f(y) \ = \ \langle \nabla f(y), g(y)\rangle+\frac{1}{2} \langle \sigma\sigma^{\prime}, \nabla^2 f(y) \rangle_{{\rm HS}}, \quad f \in \mathcal{D}(\mathcal{A}),
\end{eqnarray}
where $\mathcal{D}(\mathcal{A})$ is the domain of $\mathcal{A}$, whose exact form is determined by the function space where the semigroup $(P_{t})_{t \ge 0}$ is located.

As $\tl{X}_{0}^{\eta}=x$, we denote $\tl{X}_{k}^{\eta,x}$ for the $k$-th step iterative. Denote by $\tilde{\mathcal{P}}_{\eta}(x,\cdot)$ the one step transition probability for the Markov chain $(\tl{X}_k^{\eta})_{k\in \mathbb{N}_0}$ with $\tl{X}_0^{\eta}=x$, that is, for $f\in \mathcal{B}_b(\R^d,\R )$ and $x\in \R^d$, one has
\begin{eqnarray*}
\tilde{\mathcal{P}}_{\eta}f(x) \ = \  \int_{\R^d} f(y)\tl{\mathcal{P}}_{\eta}(x,\dif y),
\end{eqnarray*}
and denote $\tl{\mathcal{P}}_{\eta}^k=\tl{\mathcal{P}}_{\eta} \circ  \tl{\mathcal{P}}_{\eta}^{k-1}$ for integers $k\geq 2$.

We denote by $\E^{\mu}$ the conditional expectation given that $X_{0}$ has a distribution $\mu$. If $\mu=\delta_{x}$, we write $\E^{x}=\E^{\delta_{x}}$. ${\rm var}_{\mu}(\cdot)$, ${\rm cov}_{\mu}[\cdot,\cdot]$ and $\langle \cdot, \cdot \rangle_{\mu}$ mean the expectations are respected to $\E^{\mu}$, that is, ${\rm var}_{\mu}(A) = \E^{\mu} [|A-\E A|^2]$ , ${\rm cov}_{\mu}[A,B] = \E^{\mu} [|A-\E A|\cdot |B-\E B|]$ and $\langle A, B \rangle_{\mu} = \E^{\mu} [ \langle A, B \rangle  ]$. $\E_{\mathbb{P}}$ and $\E_{\mathbb{Q}}$ mean expectations under probability spaces $\mathbb{P}$ and $\mathbb{Q}$ respectively.

%
%
%
%

Recall that the following measure distances.

The Wasserstein-1 distance between two probability measures $\mu_1$ and $\mu_2$ is defined as (\cite[p. 2056]{HM1})
\begin{eqnarray}\label{e:dW}
d_{W}(\mu_1,\mu_2)
&=& \sup_{h \in {\rm Lip(1)}}\left\{\int h(x) \mu_1 (\dif x) - \int h(x) \mu_2 (\dif x) \right \} \nonumber   \\
&=& \sup_{h \in {\rm Lip_0(1)}} \left \{\int h(x) \mu_1 (\dif x) - \int h(x) \mu_2 (\dif x) \right \} \nonumber \\
&=& \sup_{h \in {\rm Lip(1)}} \left \{\int h(x) \mu_1 (\dif x) - \int h(x) \mu_2 (\dif x), \ \ |h(x)| \le |x| \right \}, 
\end{eqnarray}
where ${\rm Lip(1)}$ is the set of Lipschitz function with Lipschitz constant $1$, that is,  ${\rm Lip(1)}=\{ h: |h(x)-h(y)|\leq |x-y|$ for all $x,y \in \R^d \}$, and ${\rm Lip_0(1)}:=\{h \in {\rm Lip(1)}: h(0)=0\}$.

The total variation distance (\cite[p. 57]{MH2}) between two measures $\mu_1$, $\mu_2$ is defined by
\begin{eqnarray*}
\|\mu_1-\mu_2\|_{\rm{TV}}
&=& \sup_{ \substack{h \in \mathcal{B}_b(\R^d, \R), \,  \| h \|_{\infty}\leq1}  }  \left \{ \int_{\R^d} h(x) \mu_1(\dif x) - \int_{\R^d} h(x) \mu_2(\dif x) \right \} .
\end{eqnarray*}
Let $V:\R^d \rightarrow \R_+$ be a measurable function, define a weighted supremum norm on measurable functions (\cite[p. 57]{MH2}) by
\begin{eqnarray*}
\| \varphi \|_{V} &=& \sup_{x\in \R^d } \frac{ | \varphi(x) | }{ 1+ V(x) },
\end{eqnarray*}
as well as the dual norm of measures by
\begin{eqnarray*} 
\|\mu_1-\mu_2 \|_{\rm{TV}, \rm{V}}
&=& \sup \left\{  \int \varphi(x) \mu_1(\dif x) -\int \varphi(x) \mu_2(\dif x) :  \| \varphi \|_V \leq 1\right\}.
\end{eqnarray*}

An alternative expression for the weighted total variation norm is given by
\begin{eqnarray}\label{normTVV}
\|\mu_1-\mu_2\|_{\rm{TV}, \rm{V}}
&=&  \int_{\R^d} (1+V(x)) |\mu_1-\mu_2 | (\dif x),
\end{eqnarray}
where $\mu_1-\mu_2$ is a signed measure and $|\mu_1-\mu_2|$ is the absolute value of $\mu_1-\mu_2$. Under $V\ge 0$, one has the relation $\|\mu_1-\mu_2\|_{\rm{TV}} \leq \|\mu_1-\mu_2\|_{\rm{TV}, \rm{V}} $. If $1+V(x) \ge 1+c |x|^2 \ge c'|x|$ for some constants $c, c'>0$, it follows from \eqref{e:dW} and \eqref{normTVV} that there exists  a constant $C>0$ such that
\Be \label{e:dWandTV}
d_W(\mu_1,\mu_2) \ \leq \ C  \|\mu_1-\mu_2\|_{\rm{TV}, \rm{V}}.
\Ee

Let $P_t^*$ be the dual operator of $P_t$ for all $t\geq 0$, that is, for some measurable set $A$ and measure $\mu_1$, one has
\begin{eqnarray*}
(P_t^* \mu_1 )(A) \ = \ \int_{\R^d} P_t(x,A) \mu_1(\dif x).
\end{eqnarray*}

We use the letter $C$ to represent a positive constant, which may be different from line to line. Denote
\begin{eqnarray}
C_{\rm op} \ & = & \  \|R\|_{\rm op}+\|(R-\alpha I)p {\rm e}'\|_{\rm op}, \label{Cop}  \\
\tl{C}_{\rm op} \ & = & \ C_{\rm op}+\| \sigma\sigma^{\prime} \|_{\rm HS}+1+\|R-\alpha I\|_{\rm op}+|\beta|, \label{tlCop}  \\
C_m \ & = & \ 2m^2\tl{C}_{{\rm op}} {\rm \ for \ integers \ } m\geq 2. \label{Cm}
\end{eqnarray}

\section{Stein's equations and Malliavin calculus} \label{Malliavin-Stein}

We introduce in this section two Stein's equations, one for proving Theorem \ref{thm:DDE} and the other for proving Theorems \ref{thm:CLT} and \ref{thm:MDP}. The regularity of the first equation has been established by Gurvich \cite[Theorem 4.1]{Gur1} (see also \cite[Lemma 1]{BD1}), while that of the second equation will be established by Malliavin calculus because test functions $h$ are in $\mathcal{B}_b(\R^d,\R)$ instead of being Lipschitz. A weighted occupation time $L^{\e,x}_t$ captures the non-differentiability of $g$ at zero when applying Malliavin calculus.

\subsection{Stein's equation I}
We consider the following Stein's equation: for a Lipschitz function $h: \R^d \to \R $ with $\| \nabla h \|_{\infty}<\infty$,
\begin{eqnarray}\label{e:PoiLip}
\mathcal{A} f(x) \ = \ h(x) - \mu(h),
\end{eqnarray}
where $\mathcal{A}$ is defined as in \eqref{e:A}, and $\mu$ is the invariant measure for the process $(X_t)_{t\geq 0}$ in \eqref{hSDEg} with semigroup $(P_{t})_{t\geq 0}$.  Without of loss generality, we assume that $h\in {\rm Lip}_0(1)$. Then we can get the regularity for the solution to Stein's equation \eqref{e:PoiLip} from Gurvich \cite[Theorem 4.1]{Gur1} (see also \cite[Lemma 1]{BD1}).

\begin{lemma}\label{lem:Lipregf}
For $f$ in \eqref{e:PoiLip}, there exists some positive constant $C$ such that
\begin{eqnarray}
|f(x)| \ &\leq& \ C(1+|x|), \nonumber  \\
|\partial_i f(x)| \ &\leq& \ C(1+|x|^2), \nonumber  \\
|\partial_{ij} f(x)| \ &\leq& \ C(1+|x|^3), \label{e:2f} \\
\sup_{y\in \R^d: |y-x|<1} \frac{ | \partial_{ij} f(y)-\partial_{ij} f_(x)| }{|y-x|} \  &\leq& \ C(1+|x|^4),  \label{e:3f}
\end{eqnarray}
where $\nabla f=(\partial_1 f, \partial_2 f, \cdots, \partial_d f) \in \R^d$ and $\nabla^2 f= (   \partial_{ij} f )_{1\leq i,j \leq d} \in \R^{d\times d}$ are the gradient and Hessian matrix for $f$, respectively.
\end{lemma}

\subsection{Stein's equation II} 
We consider the following Stein's equation: for $h\in \mathcal{B}_b(\R^d,\R)$
\begin{eqnarray}\label{e:SE}
\mathcal{A} f(x) \ = \ h(x) - \mu(h),
\end{eqnarray}
where $\mathcal{A}$ is defined as in \eqref{e:A}, and $\mu$ is the invariant measure for the process $(X_t)_{t\geq 0}$ in \eqref{hSDEg} with semigroup $(P_{t})_{t\geq 0}$. 

\begin{lemma}\label{prop:ST}
For function $h\in \mathcal{B}_b(\R^d,\R)$, one has

(i) A solution to \eqref{e:SE} is given by
\begin{eqnarray}\label{e:SE1}
f(x) \ = \ - \int_0^{\infty} P_t[   h(x) - \mu(h) ] \dif t.
\end{eqnarray}

(ii) The solution to \eqref{e:SE} is also given by
\begin{eqnarray*} 
f(x) \ = \  \int_0^{\infty} e^{-\lambda t} P_t[  \lambda f(x) - h(x) + \mu(h) ] \dif t,  \quad  \forall  \lambda>0.
\end{eqnarray*}
\end{lemma}

\begin{lemma}\label{lem:regf}
Let $h\in \mathcal{B}_b(\R^d, \R)$ and $f$ be the solution to the Stein's equation \eqref{e:SE}. There exists some positive constant $C$ such that
	\begin{eqnarray*}
|f(x)| \ &\leq& \ C\|h\|_\infty(1+|x|^2), \\
|\nabla f(x)| \ &\leq& \ C \|h\|_{\infty}(1+ |x|^2).
	\end{eqnarray*}
\end{lemma}

We shall use the Malliavin calculus to get the regularity in Lemma \ref{lem:regf} for $h\in \mathcal{B}_b(\R^d,\R)$. Since $g(x)$ is not differentiable, we consider an approximation of SDE \eqref{hSDEg}:
\begin{eqnarray}\label{hSDEgApp}
\dif X^{\e}_t  \ = \ g_\e (X^{\e}_t) \dif t +\sigma \dif B_t,
\end{eqnarray}
where
\begin{eqnarray*}
g_{\e} (x) \ = \ -\beta p -Rx+\rho_{\e}({\rm e}^{\prime}  x) (R-\alpha I)p,
\end{eqnarray*}
and $\rho_{\e}$ is defined as below: for $0<\e<1$,
$$\rho_{\e}(y) \ = \  \begin{cases}
0, & \quad y<-\e, \\
y, & \quad y>\e, \\
\frac{3\e}{16} - \frac{ 1 }{ 16 \e^3} y^4 + \frac{ 3 }{8 \e} y^2 + \frac{1}{2} y, & \quad |y| \le \e. \\
\end{cases}$$
It is easy to check that $\|\dot{\rho}_{\e}\|_{\infty} \leq 1$ and that
\begin{eqnarray}
g_{\e} \ & \in & \ \mathcal{C}^2(\mathbb{R}^d,\mathbb{R}^d),  \nonumber  \\
\nabla g_{\e}(x) \ & = & \ -R +\dot{\rho}_{\e} ({\rm e}^{\prime}  x) (R-\alpha I)p {\rm e}^{\prime} , \label{e:Nge} \\
\nabla^2 g_{\e}(x) \ & =& \ \ddot{\rho}_{\e}({\rm e}^{\prime}  x){\rm e}\otimes(R-\alpha I)  p{\rm e}^{\prime} . \nonumber
\end{eqnarray}
Moreover, we can see that $ \lim\limits_{\e \rightarrow 0}g_{\e}(x)=g(x)$ for all $x\in \R^d$ and
$$\lim_{\e \rightarrow 0}\nabla g_{\e}(x) \ = \ \begin{cases}-R +(R-\alpha I)p {\rm e}^{\prime} , & \quad {\rm e}^{\prime} x>0,
\\
-R, & \quad {\rm e}^{\prime}  x<0,  \\
-R + \frac{1}{2}(R-\alpha I)p{\rm e}^{\prime}  , & \quad {\rm e}^{\prime} x=0.
\end{cases}$$

\begin{lemma}\label{lem:XXem}
For all $x \in \R^{d}$, $t\geq 0$ and intergers $m \geq 2$, we have
\begin{eqnarray}  \label{e:XXem}
\E |X^{\e,x}_t|^{m}, \E|X^{x}_{t}|^{m} \ &\leq& \ e^{C_m t}(|x|^{m}+1),
\end{eqnarray}
where $C_m$ is in \eqref{Cm}. 
Moreover, we have as $\e \to 0$,
\begin{eqnarray}
\E|X^{\e,x}_{t}-X^{x}_{t}|^m \rightarrow 0,  \qquad  t \geq 0. \label{e:XeCon-1}
\end{eqnarray}
\end{lemma}

We consider the derivative of $X^{\e,x}_t$ with respect to initial value $x$, which is called the Jacobi flow. Let $u \in \mathbb{R}^d$ and the Jacobi flow $J_{t}^{\e,x}$ along the direction $u$ is defined as
\begin{eqnarray*}
\nabla_{u} X^{\e,x}_{t} &=& \lim_{\e_1 \rightarrow 0} \frac{ X^{\e,x+\e_1 u}_t - X^{\e,x}_t  }{\e_1}, \quad t\geq 0.
\end{eqnarray*}
The above limit exists and satisfies
\begin{eqnarray*}
\frac{\dif}{\dif t}\nabla_{u} X^{\e,x}_{t} &=&\nabla g_{\e} (X^{\e,x}_t) \nabla_{u} X^{\e,x}_{t}, \quad \nabla_u X_{0}^{\e,x}=u.
\end{eqnarray*}
Define
\begin{eqnarray*}
J_{s,t}^{\e,x}:&=&\exp \left( \int_s^t \nabla g_{\e} (X_{r}^{\e,x}) \dif r \right), \ \ \ \ \ \ 0 \le s \le t<\infty.
\end{eqnarray*}
It is called the Jacobian between $s$ and $t$. For notational simplicity, denote $J^{\e,x}_{t}=J^{\e,x}_{0,t}$. Then we have
\begin{eqnarray}\label{hJaF1}
\nabla_{u} X^{\e,x}_{t}   &=&  J_{t}^{\e,x} u .
\end{eqnarray}
Define
$$\widetilde{\nabla g(x)}:\ = \ -R +1_{\{{\rm e}^{\prime} x>0\}} (R-\alpha I)p {\rm e}^{\prime}.$$
It is easy to see that $\lim_{\e \rightarrow 0} \nabla g_{\e}(x)=\widetilde{\nabla g(x)}$ for all ${\rm e}^{\prime} x \ne 0$.
Because $g(x)$ is not differentiable for ${\rm e}^{\prime}  x = 0$, it is necessary for us to define the above $\widetilde{g(x)}$ which takes the same value as $\nabla g(x)$ for ${\rm e}^{\prime} x \ne 0$ and has a definition on ${\rm e}^{\prime}  x=0$. Define
$$J^{x}_{s,t}:\ = \ \exp\left(\int_{s}^{t} \widetilde{\nabla g(X^{x}_{r})} \dif r\right), \ \ \ \ \ \ x \in \R^{d}, \  \ 0 \le s \le t<\infty.$$
Then we have the following lemma. 
\begin{lemma}  \label{l:XeCon}
For any $x \in \R^{d}$, as $\e \rightarrow 0$, the following relation holds
\begin{eqnarray*}
 \|J_{s,t}^{\e,x}-J^{x}_{s,t}\|_{{\rm op}}\ {\longrightarrow} \ 0, \ \ \ \ \ \ \ & & 0 \le s \le t<\infty, \ \ \ \ \ { \rm a.s. }
\end{eqnarray*}
\end{lemma}

Now we give estimates for  $\|J^{\e,x}_{s,t}\|_{ {\rm op} }$ and $\|J^{x}_{s,t}\|_{ {\rm op} }$. By \eqref{e:Nge}, we can easily see that
\begin{eqnarray}  \label{e:Ngeop}
\|\nabla g_{\e}(x)\|_{ {\rm op} } \ \le \ \|R\|_{ {\rm op} }+\|(R-\alpha I) p {\rm e}^{\prime} \|_{ {\rm op} } \ = \ C_{{\rm op}},
\end{eqnarray}
from which we obtain
\begin{eqnarray*}
	\|J^{\e,x}_{s,t}\|_{ {\rm op} }
	&  \le  & \exp\left(\int_s^{t} \|\nabla g_{\e}(X^{\e,x}_{r})\|_{ {\rm op} } \dif r\right) \ \le \ e^{C_{{\rm op}}(t-s)}.
\end{eqnarray*}
So for all $0 \le s \le t<\infty$, we have
\begin{eqnarray}\label{e:JJe}
\|J^{x}_{s,t}\|_{ {\rm op} },\|J^{\e,x}_{s,t}\|_{ {\rm op} } \  \le \ e^{C_{{\rm op}}(t-s)},
\end{eqnarray}
where the bound of $\|J^{x}_{s,t}\|_{ {\rm op} }$ comes from the same argument since the bound in \eqref{e:Ngeop} also holds for $\widetilde{\nabla g(x)}$. Observe that the above estimates immediately result that for $u\in\R^d$,
\begin{eqnarray*}
|\nabla_u X_t^x|,\,|\nabla_u X_t^{\e,x}| \ \leq  \  e^{C_{{\rm op}}t}|u|.
\end{eqnarray*}

\subsection{Bismut's formula of Malliavin calculus for SDE \eqref{hSDEgApp} }

Let $v\in L^2_{\rm{loc}} ([0,\infty) \times (\Omega, \mathcal{F}, \mathbb{P}), \mathbb{R}^d)$, that is, $\mathbb{E} \int_0^t |v(s)|^2 \dif s < \infty$ for all $t >0$. Assume that $v$ is adapted to the filtration  $ (\mathcal{F}_t)_{t\geq 0}$ with $\mathcal{F}_t= \sigma(B_s: 0\leq s \leq t)$, that is, $v(t)$ is  $\mathcal{F}_t$ measurable for $t\geq 0$. Define
\begin{eqnarray}\label{hMCV}
\mathbb{V}(t) \  = \  \int_0^t v(s) \dif s, \quad t\geq 0.
\end{eqnarray}
For $t>0$, let $F_t: \mathcal{C}([0,t],\mathbb{R}^d) \rightarrow \mathbb{R}^d$ be a $\mathcal{F}_t$ measurable map. If the following limit exists
\begin{eqnarray*}
D_{\mathbb{V}} F_t(B) \  = \  \lim_{\e_1 \rightarrow 0} \frac{  F_t(B+\e_1 \mathbb{V}) - F_t(B)  }{\e_1 }
\end{eqnarray*}
in $L^2 (  (\Omega, \mathcal{F}, \mathbb{P}), \mathbb{R}^d  ) $, then $F_t(B)$ is said to be Malliavin differentiable and $D_{\mathbb{V}} F_t (B)$ is called the Malliavin derivative of $F_t(B)$ in the direction $\mathbb{V}$.

\textbf{Bismut's formula.}  For Malliavin differentiable $F_t(B)$ such that $F_t(B), D_{\mathbb{V}} F_t(B) \in L^2((\Omega,\mathcal{F},\mathbb{P}), \mathbb{R}^d)$, we have
\begin{eqnarray}\label{e:BisFor}
\mathbb{E} [D_{\mathbb{V}} F_t(B)]  \ = \   \mathbb{E}  \left[   F_t(B)  \int_0^t \langle  v(s), \dif B_s  \rangle \right].
\end{eqnarray}

The following Malliavin derivative of $X^{\e,x}_t$ along the direction $\mathbb{V}$ exists in $L^2((\Omega,\mathcal{F},\mathbb{P}), \mathbb{R}^d)$ and is defined by
\begin{eqnarray}\label{MaC1}
D_{\mathbb{V}} X^{\e,x}_t &=& \lim_{\e_1 \rightarrow 0 }\frac{  X^{\e,x}_t(B+\e_1 \mathbb{V}) - X^{\e,x}_t(B)  }{\e_1}.
\end{eqnarray}
It satisfies the following equation
\begin{eqnarray*}
D_{\mathbb{V}} X^{\e,x}_t &=& \sigma \mathbb{V}(t)  + \int_0^t \nabla g_{\e} (X_{s}^{\e,x}) D_{\mathbb{V}} X_{s}^{\e,x} \dif s, \quad D_{\mathbb{V}} X_{0}^{\e,x}=0,
\end{eqnarray*}
which is solved by
\begin{eqnarray*}
D_{\mathbb{V}} X^{\e,x}_t &=&  \int_0^t J_{r,t}^{\e,x} \sigma v(r) \dif r.
\end{eqnarray*}
Taking $v(r)=\frac{\sigma^{-1}}{ t}  J^{\e,x}_{r}u$ for $0\leq r \leq t$, by \eqref{hJaF1}, we get
\begin{eqnarray}  \label{e:DVNu}
D_{\mathbb{V}} X^{\e,x}_t&=& \nabla_{u} X^{\e,x}_{t} .
\end{eqnarray}
With the same $v$, a similar straightforward calculation gives that
\begin{eqnarray*}
D_{\mathbb{V}} X_{s}^{\e,x} & = &\frac{s}{t} \nabla_u X_{s}^{\e,x}, \quad 0 \le s \le t.
\end{eqnarray*}

For further use, for $x, u\in \R^d$, we define
\begin{eqnarray*}
\mathcal{I}_{u}^{\e,x}(t) \ : = \ \frac{1}{ t } \int_0^t \langle \sigma^{-1} J^{\e,x}_{r} u, \dif B_r \rangle
\textrm{ \ \ and \ \ }
\mathcal{I}_{u}^{x}(t) \ : = \ \frac{1}{ t } \int_0^t \langle \sigma^{-1} J^{x}_{r} u, \dif B_r \rangle.
\end{eqnarray*}

Now we are at the position to state the following lemmas.
\begin{lemma}\label{lem:EIm}
For all $x, u \in \R^d$, $m \ge 2$ and $t>0$, we have
\begin{eqnarray}
\E \left|\mathcal{I}_{u}^{\e,x}(t)\right|^m, \E \left|\mathcal{I}_{u}^{x}(t)\right|^m \ &\le& \ \frac{C |u|^{m}}{t^{m/2}} e^{mC_{ {\rm op} }t}, \label{e:IuexEst}  \\
	\lim_{\e \rightarrow 0} \E \left|\mathcal{I}_{u}^{\e,x}(t)- \mathcal{I}_{u}^{x}(t)\right|^m&=&0. \label{e:IueCon}
\end{eqnarray}
\end{lemma}

\begin{lemma}\label{hLef2}
Let $\psi \in \mathcal{C}^1(\mathbb{R}^d,\mathbb{R})$ be such that $\| \nabla   \psi \|_{\infty} < \infty$. For every $t>0$, $x \in \mathbb{R}^d$  and $u \in \mathbb{R}^d$, we have
\begin{eqnarray*}
\nabla_{u} \mathbb{E} [\psi (X^{x}_t)]  \ = \  \mathbb{E}[\psi(X^{x}_t) \mathcal{I}_{u}^x(t)].
\end{eqnarray*}
\end{lemma}

\section{Proof of Theorem \ref{thm:DDE} } \label{sec-proof strategy}

We give the proof of Theorem \ref{thm:DDE} by the help of Proposition \ref{p:GeneralErgodicEM} below, whose proof will be given in Appendix \ref{App:GeneralErgodicEM}. 
\begin{proposition} \label{p:GeneralErgodicEM}
$(\tl{X}_{k}^{\eta})_{k\in \mathbb{N}_0}$ in \eqref{e:XD} admits a unique invariant measure $\tl \mu_{\eta}$ and is exponentially ergodic. More precisely, for any $k\in \mathbb{N}_0$, one has
\begin{eqnarray*}
d_W( (\tl{\mathcal{P}}_{\eta}^k)^* \nu, \tl{\mu}_{\eta})
\ &\leq& \ C\eta^{-1} e^{-c k\eta} , \\
\| (\tl{\mathcal{P}}_{\eta}^k)^* \nu- \tl{\mu}_{\eta} \|_{\rm TV}
\ &\leq& \ C\eta^{-1} e^{-c k\eta},
\end{eqnarray*}
where $C, c$ are positive constants independent of $k$ and $\eta$. Moreover, for integers $\ell \geq 2$, there exists some positive constant $C$ depending on $\ell$ but not on $\eta$ such that
\begin{eqnarray*}
\tl \mu_{\eta}(|\cdot|^{\ell}) \ \ \leq \ \ C, \ \ \ \ \ell \ge 2.
\end{eqnarray*}
\end{proposition}

\begin{proof}[{\bf Proof of Theorem  \ref{thm:DDE}}]
(i) By Proposition \ref{p:GeneralErgodicEM}, we know $(\tl{X}_{k}^{\eta})_{k\in \mathbb{N}_0}$ in \eqref{e:XD} admits a unique invariant measure $\tl{\mu}_{\eta}$. Let the initial value $\tl{X}^{\eta}_{0}$ take the invariant measure $\tl \mu_{\eta}$.
We know $(\tl{X}^{\eta}_{k})_{k \in \mathbb{N}_0}$ is a stationary Markov chain. Denote $W=\tilde{X}_{0}^{\eta}$, $W'=\tilde{X}_{1}^{\eta}$ and $\delta = W'-W$. It is easy to see that
\begin{eqnarray}  \label{e:EDelW}
\mathbb{E}[\delta|W] \ = \ g(W) \eta \text{ \ \ and \ \ }
\mathbb{E}[\delta \delta^{\prime}|W] \  =\ g(W) g^{\prime}(W) \eta^2 +\sigma \sigma^{\prime}\eta.
\end{eqnarray}

Let $f$ be the solution to Eq. \eqref{e:PoiLip} with $h\in {\rm Lip}_0(1)$. Since $W$ and $W'$ have the same distribution, we have
\begin{eqnarray}\label{hMC51e}
0&=& \mathbb{E}f(W')-\mathbb{E}f(W) \nonumber \\
&=& \mathbb{E}[ \langle \delta, \nabla f(W) \rangle  ]
+\mathbb{E} \int_0^1 \int_0^1 r \langle \delta \delta^{\prime}, \nabla^2 f (W+\tl{r} r \delta ) \rangle_{\textrm{HS}} \dif \tl{r} \dif r \nonumber \\
&=&\mathbb{E}[  \langle g(W),\nabla f(W) \rangle ]\eta+\mathbb{E} \int_0^1 \int_0^1 r \langle \delta \delta^{\prime}, \nabla^2 f(W+\tl{r} r \delta ) \rangle_{\textrm{HS}} \dif \tl{r} \dif r, \nonumber
\end{eqnarray}
where
\begin{eqnarray}\label{hMC52e}
\mathbb{E}[ \langle \delta, \nabla f(W) \rangle  ]
\ = \ \mathbb{E}[ \langle \mathbb{E}(\delta|W),\nabla f(W)\rangle  ]
\ = \ \mathbb{E}[  \langle g(W),\nabla f(W) \rangle ]\eta. \nonumber
\end{eqnarray}
In addition,
\begin{eqnarray*}
&& \mathbb{E} \int_0^1 \int_0^1 r \langle \delta \delta^{\prime}, \nabla^2 f(W+\tl{r} r \delta ) \rangle_{\textrm{HS}} \dif \tl{r} \dif r \nonumber \\
&=& \frac{1}{2}\mathbb{E} \langle \delta \delta^{\prime}, \nabla^2 f(W) \rangle_{\textrm{HS}}+
\mathbb{E} \int_0^1 \int_0^1 r \langle \delta \delta^{\prime}, \nabla^2 f(W+\tl{r} r \delta )-\nabla^2 f(W) \rangle_{\textrm{HS}} \dif \tl{r} \dif r \nonumber \\
&=& \frac{\eta}{2} \mathbb{E} [ \langle \sigma \sigma^{\prime}, \nabla^2 f(W) \rangle_{\textrm{HS}}  ] +\frac{\eta^2}{2} \mathbb{E} [  \langle g(W) g^{\prime}(W), \nabla^2 f(W) \rangle_{\textrm{HS}} ] \nonumber \\
&\ & \ \ \ \ \ \ \ \ \ \ \ \ \ \ \ \ \ \ \ +\mathbb{E} \int_0^1 \int_0^1 r \langle \delta \delta^{\prime}, \nabla^2 f (W+\tl{r} r \delta )-\nabla^2 f(W) \rangle_{\textrm{HS}} \dif \tl{r} \dif r,
\end{eqnarray*}
where the second equality is by the relation $\mathbb{E}[\langle \delta \delta^{\prime}, \nabla^2 f(W) \rangle_{\textrm{HS}}]=\mathbb{E}[\langle \E[\delta \delta^{\prime}|W], \nabla^2 f(W) \rangle_{\textrm{HS}}]$ and \eqref{e:EDelW}.
Collecting the previous relations, we obtain
\begin{eqnarray}\label{e:Afh}
\mathbb{E}[\mathcal{A}f(W) ] \ = \ \frac \eta 2 {\rm I}+\frac{1}{\eta} {\rm II},
\end{eqnarray}
where
$$
{\rm I} \ = \ -\mathbb{E} [ \langle g(W) g^{\prime}(W), \nabla^2 f(W) \rangle_{\textrm{HS}} ], \ \
{\rm II} \ = \ -\mathbb{E} \int_0^1 \int_0^1 r \langle \delta \delta^{\prime}, \nabla^2 f(W+\tl{r} r \delta )-\nabla^2 f(W) \rangle_{\textrm{HS}} \dif\tl{r} \dif r.
$$

Since $|g(x)| \leq \tl{C}_{\rm op}(1+|x|)$ for all $x\in \R^d$ with $\tl{C}_{{\rm op}}$ in \eqref{tlCop}, by using \eqref{e:2f} in Lemma \ref{lem:Lipregf}, one has
\begin{eqnarray}\label{e:Afh-1}
|{\rm I}|
&=&\mathbb{E}|\langle g(W) g^{\prime}(W), \nabla^2 f(W) \rangle_{\textrm{HS}} |
\ \leq \ C(1+\E|W|^5)
\ \leq \ C,
\end{eqnarray}
where the last inequality is by Proposition \ref{p:GeneralErgodicEM}.

We claim that there exists some positive constant $C$, independent of $\eta$ such that
\begin{eqnarray}
\mathbb{E} \int_0^1 \int_0^1 r \langle \delta \delta^{\prime}, \nabla^2 f(W+\tl{r} r \delta )-\nabla^2 f(W) \rangle_{\textrm{HS}} 1_{\{|\delta|< 1\} } \dif\tl{r} \dif r
\ &\leq& \ C\eta^{\frac{3}{2}}, \label{e:claim-1}   \\
\mathbb{E} \int_0^1 \int_0^1 r \langle \delta \delta^{\prime}, \nabla^2 f(W+\tl{r} r \delta )-\nabla^2 f(W) \rangle_{\textrm{HS}} 1_{ \{ |\delta|\geq 1 \} } \dif\tl{r} \dif r
\ &\leq&\  C\eta^{\frac{3}{2}},\label{e:claim-2}
\end{eqnarray}
thus, we know
\begin{eqnarray}\label{e:Afh-2}
|{\rm II}| \ \leq \ C\eta^{\frac{3}{2}}.
\end{eqnarray}

By using \eqref{e:Afh}, \eqref{e:Afh-1}, \eqref{e:Afh-2} and Eq. \eqref{e:PoiLip},    one has
\begin{eqnarray*}
d_{W}(\tl \mu_{\eta},\mu)
&=& \sup_{h\in {\rm Lip_0(1)}} |\E h(W)-\mu(h)|
\ = \ \sup_{h\in {\rm Lip_0(1)} } | \mathbb{E}\mathcal{A}f(W)|
\ \leq \ C\eta^{\frac{1}{2}}.
\end{eqnarray*}

It remains to show \eqref{e:claim-1} and \eqref{e:claim-2}. It follows from \eqref{e:3f} in Lemma \ref{lem:Lipregf} and Proposition \ref{p:GeneralErgodicEM} that
\begin{eqnarray*}
&& \mathbb{E} \int_0^1\int_0^1 r \langle \delta \delta^{\prime}, \nabla^2 f(W+\tl{r} r \delta )-\nabla^2 f(W) \rangle_{\textrm{HS}} 1_{\{|\delta|< 1\} } \dif\tl{r} \dif r   \\
&\leq& C\E[|\delta|^3 (1+|W|^4) 1_{\{|\delta|< 1\} } ] \\
&\leq& C\E[ |\eta g(W) + \sigma \eta^{\frac{1}{2}} \xi_1|^3 (1+|W|^4) ]  \\
&\leq& C\eta^{\frac{3}{2}} \E[1+|W|^7]  \\
&\leq& C\eta^{\frac{3}{2}},
\end{eqnarray*}
where the third inequality holds from small $\eta<1$.

It follows from \eqref{e:2f}  in Lemma \ref{lem:Lipregf} and Proposition \ref{p:GeneralErgodicEM} that for small $\eta<1$
\begin{eqnarray*}
&& \mathbb{E} \int_0^1 \int_0^1 r \langle \delta \delta^{\prime}, \nabla^2 f(W+\tl{r} r \delta )-\nabla^2 f(W) \rangle_{\textrm{HS}} 1_{ \{ |\delta|\geq 1 \} } \dif\tl{r} \dif r  \\
&\leq& C\E[ |\delta|^2 (1+|W|^3+|\delta|^3) 1_{ \{ |\delta|\geq 1 \}  }  ]  \\
&\leq& C\E[ |\delta|^2 1_{ \{ |\delta|\geq 1 \}  }  ]
+C\E[ |\delta|^2 |W|^3 1_{ \{ |\delta|\geq 1 \}  }  ]
+C\E[ |\delta|^5  ]  \\
&\leq& C(\E[ |\delta|^4])^{\frac{1}{2}}  \PP^{\frac{1}{2}} ( |\delta|\geq 1 )
+C (\E[|W|^6])^{\frac{1}{2}} ( \E[ |\delta|^8] )^{\frac{1}{4}} \PP^{\frac{1}{4}}( |\delta|\geq 1 )
+C\E[ |\delta|^5  ]  \\
&\leq& C\eta^{\frac{3}{2}},
\end{eqnarray*}
where the last inequality holds from Chebyshev's inequality and Proposition \ref{p:GeneralErgodicEM}, that is,
\begin{eqnarray*}
\PP( |\delta|\geq 1)
\ \leq \ \E|\delta|^k
\ \leq \ C\eta^{\frac{k}{2}} \E[1+|W|^k]
\ \leq \ C \eta^{\frac{k}{2}}
\end{eqnarray*}
for any integers $k\geq 1$.

(ii) By triangle inequality and using Proposition \ref{p:GeneralErgodicEM}, one has
\begin{eqnarray*}
d_W( \mathcal{L}(\tl{X}_N^{\eta}),\mu)
&\leq& d_W( \mathcal{L}(\tl{X}_N^{\eta}),\tl{\mu}_{\eta})
+d_W(\tl{\mu}_{\eta},\mu)
\ \leq \  C\eta^{-1} e^{-c N\eta} + C\eta^{\frac{1}{2}}.
\end{eqnarray*}
Taking $\eta=\delta^2$ and $N:=O(\delta^{-2} \log \delta^{-1})$, one has $\eta^{-1} e^{-cN\eta} \leq \delta$, it implies that
\begin{eqnarray*}
d_W( \mathcal{L}(\tl{X}_N^{\eta}),\mu)
\ \leq \ \delta.
\end{eqnarray*}
The proof is complete.
\end{proof}

\begin{remark}
Recall the diffusion steady-state approximation of the same model was studied in \cite{BD1}, the diffusion $X_t$ can be regarded as a diffusion approximation of the diffusion-scaled process $\hat{X}^n(t) = \frac{X^n(t) - n\gamma}{\sqrt{n}}$, which admits a unique ergodic measure $\hat \mu^n$ from \cite{Gur1}. The result in  \cite{BD1} implies
\begin{equation}  \label{e:BrDa-0}
d_W(\hat \mu^n, \mu) \ \leq  \ \frac{C}{\sqrt{n}}.
\end{equation}
We apply the EM scheme developed in this paper, by Theorem \ref{thm:DDE} (i) and \eqref{e:BrDa-0}, we have
$$
d_W(\hat \mu^n, \tilde{\mu}_\eta) \ \leq \ C (n^{-1/2} + \eta^{1/2}) \ \leq \ Cn^{-1/2},
$$
as $\eta$ is sufficiently small (say $\eta=n^{-1}$). Moreover, by Theorem \ref{thm:DDE} (ii), we can only run the EM scheme $N=O(n \log n)$ steps, we can obtain
$$d_W( \mathcal{L}(\tl{X}_N^{\eta}),\hat \mu^n) \ \leq \ C n^{-1/2}.$$
\end{remark}

\section{Moment estimate for weighted occupation time} \label{app:occupation}

 We introduce in this section a weighted occupation time and study its moment estimate,.
 This result is used in the proof of Lemma \ref{l:XeCon}, and we expect that the method in this section may be used in future research. 

The weighted occupation time $L_t^{\e,x}$ is defined as  
\begin{eqnarray*}
L_t^{\e,x} &=& \int_0^t \left[-\frac{1}{\e^2} ({\rm e}^{\prime} X_s^x)^2 +1 \right] 1_{ \{ |{\rm e}^{\prime} X_s^x| \leq \e \} } \dif s.
\end{eqnarray*}
We know $L_t^{\e,x} \geq 0$ for all $t\geq 0$. We call $L_t^{\e,x}$ weighted occupation time because they can all be represented as an integral over an occupation measure. We have
\begin{eqnarray*}
L^{\e,x}_t & = & \int_0^t \psi({\rm e}' X^x_s) 1_{\{|{\rm e}' X^x_s| \le \e\}} \dif s  \ = \  \int_{|y| \le \e} \psi(y) A^{x}_t(\dif y),
\end{eqnarray*}
where $\psi(y)=1-\frac{y^2}{\e^2} $ and $A^{x}_t(\cdot):=\int_0^t  \delta_{  {\rm e}' X^{x}_s}   (\cdot) \dif s$ with $\delta_z(\cdot)$ being the delta function of a given $z$. $A^x_t(\cdot)$ is called the occupation measure of ${\rm e}' X^x_s$ over time $[0,t]$, if $\psi(x)=1$, then $L^{\e,x}_t$ will be occupation time of $({\rm e}' X^x_s)_{0 \le s \le t}$ on the set $\{|y| \le \e\}$.

\begin{proposition}\label{lem:occupation}
For $L_t^{\e,x}$ defined above, there exists some positive constant $C$ such that
\begin{eqnarray*}
\E L_t^{\e,x} &\leq&  C\e e^{\frac{C_2}{2} t} (1+|x|)(1+t).
\end{eqnarray*}
\end{proposition}

\begin{proof}[Proof of Proposition \ref{lem:occupation}]
%
%
%

Denote 
\begin{eqnarray}\label{e:rhobar0}
\phi_{\e}(y)=
\left\{
\begin{array}{lllll}
\frac{2}{3}\e y -\frac{1}{4}\e^2,& \text{if } y>\e,\\
-\frac{1}{12\e^2}y^4 + \frac{1}{2}y^2,  & \text{if }   -\e\leq y\leq \e, \\
-\frac{2}{3}\e y-\frac{1}{4}\e^2, & \text{if } y <-\e.
\end{array}
\right.
\end{eqnarray}

It is easy to check that
$$\ddot{\phi}_{\e}(y) \ = \ \left[ \frac{-1}{\e^2}y^2+1 \right]1_{ \{ |y|\leq \e \}  },$$
and 
$|\phi_{\e}(y)|\leq C\e |y|+C\e^2$, $|\dot{\phi}_{\e}(y)|\leq C\e$ and $|\ddot{\phi}_{\e}(y)|\leq 1$ for all $y\in \R$ and the positive constant $C$ is independent of $\e$. 

Applying It\^{o}'s formula to the function $\phi_{\e}$, one has
\begin{eqnarray*}
\phi_{\e}( {\rm e^{\prime}} X_t^x)
&=& \phi_{\e}( {\rm e^{\prime} } x)
+\int_0^t \dot{\phi}_{\e}({\rm e^{\prime}} X_s^x) {\rm e^{\prime}}g(X_s^x) \dif s + \int_0^t \dot{\phi}_{\e}( {\rm e^{\prime}}X_s^x ) {\rm e^{\prime}}\sigma \dif B_s + \frac{|\sigma^{\prime} {\rm e}|^2}{2} L_t^{\e,x},
\end{eqnarray*}
which implies that
\begin{eqnarray*}
\E L_t^{\e,x}  &\leq& C \E |\phi_{\e}( {\rm e^{\prime}} X_t^x)|
+|\phi_{\e}( {\rm e^{\prime} } x)|
+\E \left|\int_0^t \dot{\phi}_{\e}({\rm e^{\prime}} X_s^x) {\rm e^{\prime}}g(X_s^x) \dif s \right|  \\
& \leq& C\e e^{\frac{C_2}{2}t}(1+|x|)(1+t),
\end{eqnarray*}
where the last inequality holds from Lemma \ref{lem:XXem}. The proof is complete.
\end{proof}

\section{Proofs of Theorems \ref{thm:CLT} and \ref{thm:MDP}} \label{s:CLTMDP}
  We firstly prove Theorem \ref{thm:CLT} by It\^{o}'s formula and martingale CLT, and then prove Theorem \ref{thm:MDP} by a criterion by Wu \cite{WLM1}.
\begin{proof}[Proof of Theorem \ref{thm:CLT}]
For $(X_t)_{t\geq 0} $ in SDE \eqref{hSDEg} with $X_0=x$, by using It\^{o}'s formula to $f$ which is the solution to Stein's equation \eqref{e:SE}, we have
\begin{eqnarray*}
f(X_t^x)-f(x)
&=&\int_0^t \mathcal{A}f(X_s^x)\dif s+\int_0^t  (\nabla f(X_s^x))^{\prime} \sigma  \dif B_s  \nonumber \\
&=&\int_0^t [ h(X_s^x)-\mu(h)] \dif s+\int_0^t  (\nabla f(X_s^x))^{\prime} \sigma \dif B_s.
\end{eqnarray*}
	It implies that
	\begin{eqnarray*}
		\sqrt{t} \left[\frac{1}{t}\int_0^t \delta_{X_s^{x}}(h)\dif s-\mu(h) \right]
		&=&\frac{1}{\sqrt{t}}\int_0^t [ h(X_s^{x}) -\mu(h) ] \dif s  \\
		&=&\frac{1}{\sqrt{t}} [ f(X_t^x)-f(x)]-\frac{1}{\sqrt{t}} \int_0^t (\nabla f(X_s^x))^{\prime} \sigma \dif B_s.
	\end{eqnarray*}
	It follows from Lemma \ref{lem:regf} and estimate for $\E V(X_t^x)$ in \eqref{e:Vm} that 
	\begin{eqnarray*}
	\E \left| \frac{1}{\sqrt{t}} [ f(X_t^{x})-f(x)] \right|
	&\to& 0 { \ \ \rm as \ \ } t \to \infty.
	\end{eqnarray*}

It follows from Lemmas \ref{lem:regf}, \ref{lem:AV2} and  \eqref{e:BV}, that  for some $C>0$ such that
\begin{eqnarray*}
\mu( |\sigma^{\prime} \nabla f|^2  ) \ \leq \ C\mu(V^4)  \ < \ \infty .
\end{eqnarray*}

We write
	\begin{eqnarray*}
\frac{1}{\sqrt{t}} \int_0^t (\nabla f(X_s^x))^{\prime} \sigma \dif B_s
\ = \ 
\frac{1}{\sqrt{t}} \left(\int_0^1 + \int_1^2 +\cdots + \int_{\lfloor t \rfloor-1}^{\lfloor t \rfloor} + \int_{\lfloor t \rfloor}^t \right) (\nabla f(X_s^x))^{\prime} \sigma \dif B_s,
	\end{eqnarray*}
and denote
	\begin{eqnarray*}
U_i \ = \  \int_{i-1}^i (\nabla f(X_s^x))^{\prime} \sigma \dif B_s \text{ \ \ for \ \ } i=1,2,\cdots,\lfloor t \rfloor
	\end{eqnarray*}
and
	\begin{eqnarray*}
U_{\lfloor t \rfloor+1} \ = \  \int_{\lfloor t \rfloor}^t (\nabla f(X_s^x))^{\prime} \sigma \dif B_s.
	\end{eqnarray*}
We know $U_i$ are martingale differences and $\E U_i^2 <\infty$ for all $i=1,2,\cdots,\lfloor t \rfloor+1$. We claim that
\begin{eqnarray}
\lim_{t \to \infty} \E\left( \max_{1\leq i \leq \lfloor t \rfloor +1 } \frac{1}{\mu( |\sigma^{\prime} \nabla f|^2 ) t } |U_i|^2 \right) \ = \ 0, \label{e:CLT1} \\
\lim_{t \to \infty} \E \left| \sum_{i=1}^{\lfloor t \rfloor +1 } \frac{1}{\mu( |\sigma^{\prime} \nabla f|^2 ) t } |U_i|^2 -1 \right|^2 \ = \ 0. \label{e:CLT2}
\end{eqnarray}

We know \eqref{e:CLT1} and \eqref{e:CLT2} imply
\begin{eqnarray*}
\E\left( \max_{1\leq i \leq \lfloor t \rfloor +1 } \frac{1}{ \sqrt{\mu( |\sigma^{\prime} \nabla f|^2 ) t} } |U_i| \right) \to 0
\text{ \ \ and \ \ }
\sum_{i=1}^{\lfloor t \rfloor +1 } \frac{1}{\mu( |\sigma^{\prime} \nabla f|^2 ) t } |U_i|^2 \stackrel{p}{\longrightarrow} 1
	\end{eqnarray*}
as $t$ goes to infinity.  By using the martingale CLT in \cite[Theorem 2]{SS1} due to \cite{MDL1}, one has
\begin{eqnarray*}
\frac{1}{\sqrt{\mu(|\sigma^{\prime}  \nabla f|^2) t}}\sum_{i=1}^{\lfloor t \rfloor+1} U_i
\ = \ \frac{1}{\sqrt{\mu(|\sigma^{\prime}  \nabla f|^2) t}} \int_0^t  (\nabla f(X_s^x))^{\prime} \sigma \dif B_s  \ \Rightarrow  \ \mathcal{N} (0,1) \text{ \ \ as \ \ } t \to \infty.
	\end{eqnarray*}

Then, we know
	\begin{eqnarray*}
		\sqrt{t} \left[ \frac{1}{t}\int_0^t \delta_{X_s^{x}}(h)\dif s-\mu(h) \right]
		\ \Rightarrow \ \mathcal{N} (0,\mu(|\sigma^{\prime}  \nabla f|^2))
\text{ \ \ as \ \ } t \to \infty.
	\end{eqnarray*}

It reminders to show that \eqref{e:CLT1} and \eqref{e:CLT2}. For \eqref{e:CLT1}, one has
\begin{eqnarray*} 
\E\left( \max_{1\leq i \leq \lfloor t \rfloor +1} |U_i|^2 \right)
&=& \E\left( \max_{1\leq i \leq \lfloor t \rfloor +1} (|U_i|^2 1_{ \{ |U_i|^2\leq \sqrt{t} \} } + |U_i|^2 1_{ \{|U_i|^2>\sqrt{t}  \}})  \right)  \nonumber \\
&\leq& \E\left( \max_{1\leq i \leq \lfloor t \rfloor+1} |U_i|^2 1_{ \{ |U_i|^2\leq \sqrt{t} \} } \right)
+ \E\left( \max_{1\leq i \leq \lfloor t \rfloor+1} |U_i|^2 1_{ \{|U_i|^2> \sqrt{t}  \}}  \right) \nonumber \\
&\leq& \sqrt{t} + \E \left( \max_{1\leq i \leq \lfloor t \rfloor +1} |U_i|^2 1_{ \{|U_i|^2> \sqrt{t} \}}  \right) \nonumber \\
&\leq& \sqrt{t} +(\lfloor t \rfloor+1) \max_{1\leq i \leq \lfloor t \rfloor+1} \E (|U_i|^2 1_{ \{|U_i|^2> \sqrt{t} \}} ),
\end{eqnarray*}
where the last inequality holds from that 
\begin{eqnarray*}
\E \left( \max_{1\leq i \leq \lfloor t \rfloor +1} |U_i|^2 1_{ \{|U_i|^2> \sqrt{t} \}}  \right)  \ &\leq& \  \sum_{i=1}^{ \lfloor t \rfloor +1} \E (|U_i|^2 1_{ \{|U_i|^2> \sqrt{t} \}} ).
\end{eqnarray*}

Thus, we obtain
\begin{eqnarray*}
\E \left( \max_{1\leq i \leq \lfloor t \rfloor+1 } \frac{1}{ \mu( |\sigma^{\prime} \nabla f|^2 ) t} | U_i |^2 \right)
&\leq& \frac{1}{ \mu( |\sigma^{\prime} \nabla f|^2 ) t} \left[ \sqrt{t} +(\lfloor t \rfloor+1) \max_{1\leq i \leq \lfloor t \rfloor +1} \E (|U_i|^2 1_{ \{|U_i|^2> \sqrt{t} \}} ) \right] \\
&\to & 0 \text{ \ \ as \ \ } t\to \infty.
\end{eqnarray*}

For \eqref{e:CLT2}, we have
\begin{eqnarray*}
&& \E \left| \sum_{i=1}^{ \lfloor t \rfloor +1 } \frac{1}{\mu( |\sigma^{\prime} \nabla f|^2 ) t } |U_i|^2 -\frac{ \lfloor t \rfloor +1 }{t} \right|^2  
\  =  \ \E \left| \frac{1}{t} \sum_{i=1}^{ \lfloor t \rfloor +1 } \left( \frac{1}{\mu( |\sigma^{\prime} \nabla f|^2 )  } |U_i|^2 - 1 \right) \right|^2 \\
&=& \frac{1}{t^2} \sum_{i=1}^{ \lfloor t \rfloor +1 }\E\left( \frac{1}{\mu( |\sigma^{\prime} \nabla f|^2 )  } |U_i|^2 - 1 \right)^2+\frac{2}{t^2}\sum_{i<j}\E\left[ ( \frac{1}{\mu( |\sigma^{\prime} \nabla f|^2 )  } |U_i|^2 - 1 ) ( \frac{1}{\mu( |\sigma^{\prime} \nabla f|^2 )  } |U_j|^2 - 1 )  \right]  \\
&=:& {\rm I} + {\rm II}.
\end{eqnarray*}

By using Burkholder-Davis-Gundy inequality, there exists some positive constant $C$ such that
\begin{eqnarray*}
\E|U_i|^4
&=& \E \left| \int_{i-1}^i (\nabla f(X_s^x))^{\prime} \sigma \dif B_s \right|^4
\ \leq \ C \E \left| \int_{i-1}^i |(\nabla f(X_s^x))^{\prime} \sigma|^2 \dif s \right|^2.
\end{eqnarray*}
Combining with Lemma \ref{lem:regf}, we know
\begin{eqnarray*}
\E\left( \frac{1}{\mu( |\sigma^{\prime} \nabla f|^2 )  } |U_i|^2 - 1 \right)^2
&=& \E \left[ \frac{1}{\mu^2(|\sigma^{\prime} \nabla f|^2)} |U_i|^4 - \frac{2}{\mu( |\sigma^{\prime}\nabla f|^2)}|U_i|^2+1 \right] \\
&\leq & C(1+\E |X_i^x|^8).
\end{eqnarray*}
Combining this with \eqref{e:Vm} in Lemma \ref{lem:AV2} and  \eqref{e:BV}, we know
\begin{eqnarray*}
{\rm I}
\ &=& \ \frac{1}{t^2} \sum_{i=1}^{ \lfloor t \rfloor +1 } \E\left( \frac{1}{\mu( |\sigma^{\prime} \nabla f|^2 )  } |U_i|^2 - 1 \right)^2  
\ \leq \  \frac{C}{t^2} \sum_{i=1}^{ \lfloor t \rfloor +1 }  (1+\E |X_i^x|^8)
\to 0, {\ \ \rm as \ \ } t \to \infty.
\end{eqnarray*}

We also have
\begin{eqnarray*}
{\rm II}
&=&  \frac{2}{t^2}\sum_{i<j}\E \left[ \left( \frac{1}{\mu( |\sigma^{\prime} \nabla f|^2 ) } |U_i|^2 - 1 \right) \left( \frac{1}{\mu( |\sigma^{\prime} \nabla f|^2 )  } |U_j|^2 - 1 \right) \right] \\
&=& \frac{2}{t^2}\sum_{i=1}^{\lfloor t \rfloor } \sum_{j=i+1}^{\lfloor t \rfloor +1} \E \left\{ \left(\frac{1}{\mu(|\sigma^{\prime} \nabla f|^2)}|U_i|^2-1\right) \E \left[ \left ( \frac{1}{\mu( |\sigma^{\prime} \nabla f|^2 )  } |U_j|^2 - 1 \right) |\mathcal{F}_i  \right]  \right\} \\
&=& \frac{2}{\mu( |\sigma^{\prime} \nabla f|^2 ) t^2}\sum_{i=1}^{\lfloor t \rfloor } \sum_{j=i+1}^{\lfloor t \rfloor +1} \E \left\{ \left(\frac{1}{\mu(|\sigma^{\prime} \nabla f|^2)}|U_i|^2-1\right) \E[(|U_j|^2 - \mu( |\sigma^{\prime} \nabla f|^2 )  ) |\mathcal{F}_i  ] \right\} \\
&=& \frac{2}{\mu( |\sigma^{\prime} \nabla f|^2 ) t^2}\sum_{i=1}^{\lfloor t \rfloor } \sum_{j=i+1}^{\lfloor t \rfloor +1} \E\left\{ \left(\frac{1}{\mu(|\sigma^{\prime} \nabla f|^2)}|U_i|^2-1\right) \int_{j-1}^j [ \E|\sigma^{\prime} \nabla f(X_s^{X_i^x})|^2 - \mu( |\sigma^{\prime} \nabla f|^2 ) ] \dif s \right\} \\
&=& \frac{2}{\mu( |\sigma^{\prime} \nabla f|^2 ) t^2}\sum_{i=1}^{\lfloor t \rfloor } \E \left\{ \left(\frac{1}{\mu(|\sigma^{\prime} \nabla f|^2)}|U_i|^2-1\right) \int_{i}^{\lfloor t \rfloor+1} [ \E|\sigma^{\prime} \nabla f(X_s^{X_i^x})|^2 - \mu( |\sigma^{\prime} \nabla f|^2 ) ] \dif s \right\}.
\end{eqnarray*}

It follows from Lemma \ref{lem:AV2} that 
\begin{eqnarray*}
{\rm II} &\leq& \frac{C}{t^2}\sum_{i=1}^{\lfloor t \rfloor } \E \left\{ \left|\frac{1}{\mu(|\sigma^{\prime} \nabla f|^2)}|U_i|^2-1 \right| \int_{0}^{\lfloor t \rfloor+1} (1+V^2(X_i^x))e^{-cs} \dif s \right\} \\
&\leq& \frac{C}{t^2}\sum_{i=1}^{\lfloor t \rfloor } \E \left\{ \left|\frac{1}{\mu(|\sigma^{\prime} \nabla f|^2)}|U_i|^2-1 \right| (1+V^2(X_i^x))  \right\} \\
&\leq& \frac{C}{t^2}\sum_{i=1}^{\lfloor t \rfloor } \left[\E \left|\frac{1}{\mu(|\sigma^{\prime} \nabla f|^2)}|U_i|^2-1\right|^2\right]^{\frac{1}{2}} [1+\E V^4(X_i^x)]^{\frac{1}{2}} \\
&\to& 0, {\ \ \rm as \ \ } t \to \infty.
\end{eqnarray*}

Combining the estimates for ${\rm I}$ and ${\rm II}$, we know
\begin{eqnarray*}
\lim_{t \to \infty} \E \left| \sum_{i=1}^{ \lfloor t \rfloor +1 } \frac{1}{\mu( |\sigma^{\prime} \nabla f|^2 ) t } |U_i|^2 -\frac{ \lfloor t \rfloor +1 }{t} \right|^2  \  =  \ 0,
\end{eqnarray*}
thus, \eqref{e:CLT2} holds.  The proof is complete.
\end{proof}

\begin{proof}[Proof of Theorem \ref{thm:MDP}]
	It follows from Lemma \ref{lem:AV2} that for any function $\tl{f}$ satisfying $\tl{f}(x)\leq 1+V(x)$ for all $x\in \R^d$ with $V$ in \eqref{e:Lypfun}, there exist positive constants $C$ and $c$ such that
\begin{eqnarray*}
| P_t \tl{f}(x) - \mu(\tl{f}) |  \ \leq \ C(1+V(x)) e^{-c t}.
	\end{eqnarray*}
	It follows from \cite[Remark (2.17)]{WLM2} that $P_t$ has a spectral gap near its largest eigenvalue $1$ which implies that $1$ is an isolate eigenvalue. Since $P_t c = c$ for all $t>0$, one has the property that $1$ is an eigenvalue of $P_t$ for all $t>0$. If there exists some function $\hat{f}$ satisfying $0\neq \hat{f}(x) \leq 1+V(x)$ for all $x\in \R^d$ such that  $P_{t_1} \hat{f} = \lambda \hat{f}$ for some $t_1 > 0$ and $\lambda \in \mathbb{C}$ with $|\lambda|=1$, then $\lambda^{\tl{n}} \hat{f} = P_{\tl{n} t_1} \hat{f} \to \mu(\hat{f})$ as $\tl{n} \to \infty$, so that $\hat{f}$ has to be constant and $\lambda=1$. Thus, $1$ is a simple and the only eigenvalue with modulus $1$ for $P_t$.

Since $h\in \mathcal{B}_b(\R^d, \R)$, it follows from \cite[Theorem 2.1]{WLM1} that
$$\PP\left( \frac{1}{a_t\sqrt{t}}\int_0^t [ h(X_s^x) - \mu(h)  ]  \dif s \in \cdot \right)$$
satisfies the large deviation principle with speed $a_t^{-2}$ and rate function $I_h(z)=\frac{z^2}{2\mathcal{V}(h)}$ with 
$$\mathcal{V}(h) \ = \ 2\int_0^{\infty} \langle P_t h, h-\mu(h) \rangle_{\mu} \dif t,$$
that is,
\begin{eqnarray*}
	-\inf_{z \in A^{ {\rm o} } } I_h (z)
&\leq& \liminf_{t\to\infty}\frac{1}{a_t^2}\log \mathbb{P} \left( \frac{\sqrt{t}}{a_t}  \left[\mcl E^x_t(h)-\mu(h)\right] \in A \right)  \\
&\leq& \limsup_{t\to\infty}\frac{1}{a_t^2}\log \mathbb{P} \left( \frac{\sqrt{t}}{a_t}  \left[\mcl E^x_t(h)-\mu(h)\right] \in A \right)
\ \leq  \ -\inf_{z \in \bar{A}} I_h (z),
	\end{eqnarray*}
	where $\bar{A}$ and $A^{ {\rm o} }$ are the closure and  interior of set $A$, respectively.

We claim that
\begin{eqnarray}\label{e:Vh}
\mathcal{V}(h) \ = \ \mu(|\sigma^{\prime}\nabla f|^2 ),
	\end{eqnarray}
where $f$ is the solution to Stein's equation \eqref{e:SE}. Then the desired result holds.

Now, we show that the claim \eqref{e:Vh} holds. We have
\begin{eqnarray*}
\mathcal{V}(h)
&=& 2\int_0^{\infty} \langle P_t h, h-\mu(h) \rangle_{\mu} \dif t
\ = \ 2\int_0^{\infty} \langle P_t h - \mu(h), h \rangle_{\mu} \dif t.
	\end{eqnarray*}

In addition, we know
	\begin{eqnarray*}
\frac{1}{t}\E^{\mu} \left( \int_0^t [h(X_s^x) - \mu(h)] \dif s \right)^2
&=& \frac{1}{t}\E^{\mu} \left( \int_0^t \int_0^t [h(X_s^x)-\mu(h)][h(X_u^x)-\mu(h)] \dif u \dif s  \right) \\
&=& \frac{2}{t}\E^{\mu} \left( \int_0^t \int_0^u [h(X_s^x)-\mu(h)][h(X_u^x)-\mu(h)] \dif s \dif u \right) \\
&=& \frac{2}{t} \int_0^t \int_0^u \E^{\mu} \{ [h(X_s^x)-\mu(h)] \E \{ [h(X_u^x)-\mu(h)]| X_s^x \} \} \dif s \dif u \\
&=& \frac{2}{t} \int_0^t \int_0^u \E^{\mu} \{ h(X_s^x) \E \{ [h(X_u^x)-\mu(h)]| X_s^x \} \} \dif s \dif u,
	\end{eqnarray*}
where the third equality holds from conditional probability and the last equality holds because
	\begin{eqnarray*}
\int_0^t \int_0^u \E^{\mu} \{ \mu(h) \E \{ [h(X_u^x)-\mu(h)]| X_s^x \} \} \dif s \dif u
&=& \mu(h) \int_0^t \int_0^u \E^{\mu} \{  h(X_u^x)-\mu(h) \} \dif s \dif u
=0.
	\end{eqnarray*}

Furthermore, for all $0\leq s \leq u<\infty$, one has
\begin{eqnarray*}
\E^{\mu} \{ h(X_s^x) \E \{ [h(X_u^x)-\mu(h)]| X_s^x \} \}
&=&  \E^{\mu} \{ h(X_s^x) [ P_{u-s}h(X_s^x) - \mu(h) ]   \}   \\
&=& \int_{\R^d} h(y)[ P_{u-s}h(y)-\mu(h) ] \mu(\dif y),
	\end{eqnarray*}
which implies that
	\begin{eqnarray*}
&& \frac{2}{t} \int_0^t \int_0^u \E^{\mu} \{ h(X_s^x) \E \{ [h(X_u^x)-\mu(h)]| X_s^x \} \} \dif s \dif u  \\
&=& \frac{2}{t} \int_0^t\int_0^u \int_{\R^d} h(y)[ P_{u-s}h(y)-\mu(h) ] \mu(\dif y) \dif s \dif u \\
&=& \int_{\R^d} h(y) \frac{2}{t} \int_0^t\int_0^u [ P_{u-s}h(y)-\mu(h) ]  \dif s \dif u \mu(\dif y) \\
&=& \int_{\R^d} h(y) \frac{2}{t} \int_0^t\int_0^u [ P_{\tl{s}}h(y)-\mu(h) ]  \dif \tl{s} \dif u \mu(\dif y).
\end{eqnarray*}
By using the Hospital's rule, one has
\begin{eqnarray*}
&& \lim_{t\to\infty}\int_{\R^d} h(y) \frac{2}{t} \int_0^t\int_0^u [ P_{\tl{s}}h(y)-\mu(h) ]  \dif \tl{s} \dif u \mu(\dif y) \\
&=&\lim_{t\to\infty}2 \int_{\R^d} h(y) \int_0^{t} [P_{\tl{s}} h(y) - \mu(h)] \dif \tl{s} \mu(\dif y) \\
&=& 2 \int_{\R^d} h(y) \int_0^{\infty} [P_{\tl{s}} h(y) - \mu(h)] \dif \tl{s} \mu(\dif y) = \mathcal{V}(h),
\end{eqnarray*}
that is,
\begin{eqnarray*}
\frac{1}{t}\E^{\mu} \left( \int_0^t [h(X_s^x) - \mu(h)] \dif s \right)^2
&\to& \mathcal{V}(h)
\text{ \ \ as \ \ } t \to \infty.
	\end{eqnarray*}

By using It\^{o}'s formula and Stein's equation \eqref{e:SE}, we have
\begin{eqnarray*}
f(X_t^x)-f(x)
&=&\int_0^t \mathcal{A}f(X_s^x)\dif s+\int_0^t  (\nabla f(X_s^x))^{\prime} \sigma  \dif B_s  \nonumber \\
&=&\int_0^t [ h(X_s^x)-\mu(h)] \dif s+\int_0^t  (\nabla f(X_s^x))^{\prime} \sigma \dif B_s,
\end{eqnarray*}
which implies that
\begin{eqnarray*}
\frac{1}{t}\E^{\mu} \left( \int_0^t [h(X_s^x) - \mu(h)] \dif s \right)^2
&=& \E^{\mu} \left( \frac{1}{\sqrt{t}} [ f(X_t^x)-f(x)] - \frac{1}{\sqrt{t}} \int_0^t  (\nabla f(X_s^x))^{\prime} \sigma \dif B_s \right)^2 \\
&\to& \mu(|\sigma^{\prime} \nabla f |^2) \text{ \ \ as \ \ } t \to \infty.
	\end{eqnarray*}
Thus, we know $\mathcal{V}(h) = \mu( |\sigma^{\prime} \nabla f|^2 )$. The claim \eqref{e:Vh} holds. The proof is complete.
\end{proof}

\section{ Proof of Theorems \ref{thm:EMCLT} and \ref{thm:EMMDP} } \label{sec:EMCLTMDP}

We shall prove Theorems \ref{thm:EMCLT} and \ref{thm:EMMDP} from \cite[Theorem 9]{jones2004markov}  and \cite[Theorem 2.1]{WLM1}, respectively.
\begin{proof}[Proof of Theorem \ref{thm:EMCLT}]
From Proposition \ref{p:GeneralErgodicEM}, we know the Markov chain $(\tl{X}_{k}^{\eta})_{k\in \mathbb{N}_0}$ is exponentially ergodic under total variation distance and $\tl{\mu}_{\eta}(V^{\ell})\leq C$ for any integers $\ell$ and $V$ in \eqref{e:Lypfun}. It follows from \cite[Theorem 9]{jones2004markov} that for any function $h$ satisfying $|h|\leq V^{\ell}$ with some integer $\ell$ and any initial distribution, one has 
     \begin{eqnarray*}
\sqrt{n}\left[ \mcl E_n^{\eta, \tl{X}_0^{\eta}}(h)  -  \tl{\mu}_{\eta}(h)  \right] \  \Rightarrow \ \mathcal{N}(0, \sigma_h^2)
	\end{eqnarray*}
with $\sigma_h^2$ in \eqref{e:sigmaf2}. The proof is complete.
\end{proof}

\begin{proof}[Proof of Theorem \ref{thm:EMMDP}]
It follows from the proof of Proposition \ref{p:GeneralErgodicEM} that there exist some positive constants $C$ and $c$ such that
\begin{eqnarray*}
\|\tilde{\mathcal{P}}_{\eta}^n(x,\cdot) -\tl{\mu}_{\eta} \|_{1+V} \ := \ \sup_{|f|\leq 1+V} | \tilde{\mathcal{P}}_{\eta}^n f(x)-\tl{\mu}_{\eta}(f)|
\ \leq \ C\eta^{-1} e^{-c n\eta}.
	\end{eqnarray*}
	It follows from \cite[Remark (2.17)]{WLM2} that $\tilde{\mathcal{P}}^n_{\eta}$ has a spectral gap near its largest eigenvalue $1$ which implies that $1$ is an isolate eigenvalue. Since $\tilde{\mathcal{P}}^k_{\eta} c = c$ for all $k \in \mathbb{N}_0$, one has  that $1$ is an eigenvalue of $\tilde{\mathcal{P}}^k_{\eta}$ for all $k \in \mathbb{N}_0$. If there exists some function $\hat{f}$ satisfying $0\neq \hat{f}(x) \leq 1+V(x)$ for all $x\in \R^d$ such that  $\tilde{\mathcal{P}}^{k_1}_{\eta} \hat{f} = \lambda \hat{f}$ for some $k_1 \in \mathbb{N}_0$ and $\lambda \in \mathbb{C}$ with $|\lambda|=1$, then $\lambda^{\tl{n}} \hat{f} = \tilde{\mathcal{P}}^{\tl{n} k_1}_{\eta} \hat{f} \to \mu(\hat{f})$ as $\tl{n} \to \infty$, so that $\hat{f}$ has to be constant and $\lambda=1$. Thus, $1$ is a simple and the only eigenvalue with modulus $1$ for $\tilde{\mathcal{P}}_{\eta}^n$.

Since $h\in \mathcal{B}_b(\R^d, \R)$, it follows from \cite[Theorem 2.1]{WLM1} that
$$\PP\left( \frac{\sqrt{n}}{a_n}  \left[ \mcl E_n^{\eta,x}(h) - \tl{\mu}_{\eta}(h)\right]  \in \cdot \right)$$
satisfies the large deviation principle  with speed $a_n^{-2}$ and rate function $I_h(z)=\frac{z^2}{2\mathcal{V}(h)}$ with $\mathcal{V}(h)$ in \eqref{e:EMVh}, that is,
	\begin{eqnarray*}
	-\inf_{z \in A^{ {\rm o} } } \frac{ z^2 }{2 \mathcal{V}(h)}
	\ &\leq& \ 
	\liminf_{n \to\infty}\frac{1}{a_n^2}\log \mathbb{P} \left( \frac{\sqrt{n}}{a_n}  \left[ \mcl E_n^{\eta,x}(h) - \tl{\mu}_{\eta}(h)\right]    \in A\right)   \\
	\ &\leq& \
	\limsup_{ n \to\infty}\frac{1}{a_n^2}\log \mathbb{P} \left( \frac{\sqrt{n}}{a_n}  \left[ \mcl E_n^{\eta,x}(h) - \tl{\mu}_{\eta}(h)\right]  \in A\right)
\ \leq \ -\inf_{z \in \bar{A}} \frac{ z^2 }{ 2 \mathcal{V}(h) },
	\end{eqnarray*}
	where $\bar{A}$ and $A^{ {\rm o} }$ are the closure and  interior of set $A$, respectively. The proof is complete.
\end{proof}

\begin{appendix}
\section{Proof of Proposition \ref{p:GeneralErgodicEM} }  \label{App:GeneralErgodicEM}

It follows from \cite[Theorem 1]{DG1} that there exists a positive definite matrix $\tilde{Q}=( \tilde{Q}_{ij}  )_{d\times d}$ with $\sum_{i,j=1}^d |\tl{Q}_{ij}|=1$ such that
\begin{eqnarray*}
\tilde{Q}(-R) + (-R)^{\prime} \tilde{Q} \ &<&  \ 0, \\
\tilde{Q}(-(I-p{\rm e}^{\prime})R) +(-R^{\prime}(I-{\rm e}p^{\prime}))\tilde{Q}
\ &\leq& \ 0.
\end{eqnarray*}
Recall that the function $\tl{V}\in \mathcal{C}^2(\R^d,\R_+)$ be constructed in \cite[Eq. (5.24)]{DG1}, that is,
\begin{eqnarray*} 
\tl{V}(y)&=&({\rm e}^\prime y)^2+\kappa[y-p\phi({\rm e}^\prime y)]^\prime \tilde{Q}[y-p\phi({\rm e}^\prime y)], \qquad  \forall y\in \R^d,
\end{eqnarray*}
where $\kappa$ is a positive constant and $\phi\in \mathcal{C}^2(\R, \R)$ is a real-valued function which is defined as below:
\begin{eqnarray*} 
\phi(z) &=&
\left\{
\begin{array}{lll}
z,                                   &  \text{if }  z \ge 0,  \\
-\frac{1}{2},                        & \text{if }   z \leq -1, \\
-\frac{1}{2}z^4 - z^3 + z,           &\text{if }     -1<z<0.
\end{array}
\right.
\end{eqnarray*}
In addition, we know
\begin{eqnarray*}
(\nabla \tl{V}(y) )^{\prime} &=& 2({\rm e}^{\prime} y) {\rm e}^{\prime} + 2\kappa [ y^{\prime} - p^{\prime} \phi ({\rm e}^{\prime} y) ] \tilde{Q} [ I - p {\rm e}^{\prime} \dot{\phi}({\rm e}^{\prime} y)] \text{ \ \ for all \ \ } y \in \R^d.
\end{eqnarray*}
Then there exists some positive constant $C$ such that
\begin{eqnarray}\label{e:NtlV}
|\nabla \tl{V}(y)|
\ &\leq& \ C(1+|y|) \text{ \ \ for all \ \ } y \in \R^d.
\end{eqnarray}
It follows from \cite[proof of Theorem 3]{DG1} that  there exist some positive constants $\hat{C}_1, \hat{C}_2$, $\hat{c}_1$ and $\hat{c}_2$ such that
\begin{eqnarray}\label{e:BtlV}
\hat{c}_1|y|^2 - \hat{c}_2
\  \leq \
\tl{V}(y)
\ \leq  \ \hat{C}_1|y|^2+\hat{C}_2
\text{ \ \ for all \ \ } y \in \R^d.
\end{eqnarray}
It follows from \cite[proof of Proposition 4]{DG1}, there exist some positive constants $c_1$ and $\check{c}_1$ such that
\begin{eqnarray}\label{e:AtlV}
\mathcal{A}\tl{V}(y)
\ &\leq& \ -c_1 \tl{V}(y)+\check{c}_1
\text{ \ \ for any \ \ }y \in \R^d,
\end{eqnarray}
where $\mathcal{A}$ is in \eqref{e:A}.

Let the Lyapunov function $V$ be defined as
\begin{eqnarray}\label{e:Lypfun}
V(y) &=& \tl{V}(y)+\hat{c}_2 
\ = \ ({\rm e}^\prime y)^2+\kappa[y-p\phi({\rm e}^\prime y)]^\prime \tilde{Q}[y-p\phi({\rm e}^\prime y)] + \hat{c}_2, \qquad  \forall y\in \R^d,
\end{eqnarray}
where $\hat{c}_2$ is in \eqref{e:BtlV} and
\begin{eqnarray}\label{e:BV}
\hat{c}_1|y|^2
\ \leq \
V(y)
\ \leq  \ \hat{C}_1|y|^2+\hat{C}_2+\hat{c}_2
\text{ \ \ for all \ \ } y \in \R^d.
\end{eqnarray}
Furthermore, we know
\begin{eqnarray*}
\nabla V(y) \  =  \ \nabla \tl{V}(y),
\qquad
\nabla^2 V(y) \ = \ \nabla^2 \tl{V}(y)
\text{ \ \ for all \ \ } y \in \R^d,
\end{eqnarray*}
thus, \eqref{e:NtlV} also holds for function $V$, that is,
\begin{eqnarray}\label{e:NV}
|\nabla V(y)|
&\leq& C(1+|y|) \text{ \ \ for all \ \ } y \in \R^d,
\end{eqnarray}
and $\mathcal{A} V(y)=\mathcal{A}\tl{V}(y)$ for all $y \in \R^d$.
Combining with \eqref{e:AtlV}, one has
\begin{eqnarray}\label{e:AV}
\mathcal{A} V(y)
&\leq& -c_1 (\tl{V}(y)+\hat{c}_2)+\check{c}_1 +c_1\hat{c}_2 
\ \leq \ -c_1 V(y)+\breve{c}_1
\text{ \ \ for any \ \ }y \in \R^d
\end{eqnarray}
with $\breve{c}_1=\check{c}_1 +c_1\hat{c}_2$.

\begin{lemma}\label{lem:AV2}
For $\mathcal{A}$ in \eqref{e:A}, $V$ in \eqref{e:Lypfun} and integers $\ell\geq 1$, there exist some positive constants $\breve{c}_{\ell}$ depending on $\ell$ such that
\begin{eqnarray*}
\mathcal{A} V^{\ell}(x)
&\leq& -c_1 V^{\ell}(x)+\breve{c}_{\ell},
\end{eqnarray*}
and 
\begin{eqnarray}\label{e:Vm}
\mathbb{E} V^{\ell}(X_t^{x})
\ &\leq& \ e^{-c_1 t }V^{\ell}(x)+\frac{\breve{c}_{\ell}(1-e^{-c_1t})}{c_1},
\quad \forall t\geq 0.
\end{eqnarray}
In addition, $\mu(V^{\ell})\leq \frac{\breve{c}_{\ell}}{c_1}$ and $\mu(|\cdot|^{2\ell})\leq C$ where the constant $C$ depends on $\ell$. Furthermore, for any positive integers $\ell$ and probability measure $\nu$ satisfying $\nu(V^{\ell})<\infty$, one has
\begin{eqnarray*}
d_W(P_{t}^* \nu, \mu)  \ &\leq& \ C(1+\nu(V)) e^{-c t},  \\
\| P_{t}^* \nu - \mu\|_{{\rm TV}}  \ &\leq&  \ \| P_{t}^* \nu - \mu\|_{{\rm TV, V^{\ell}}}  \ \leq  \ C(1+\nu(V^{\ell})) e^{-c t}.
\end{eqnarray*}

\end{lemma}


In order to prove Proposition \ref{p:GeneralErgodicEM}, we firstly show that $(\tl{X}_k^{\eta})_{k\in \mathbb{N}_0}$ in \eqref{e:XD} is strong Feller and irreducible and the Lyapunov condition holds, and then we can get the exponential ergodicity from \cite{DFMS1, TT1}. Rewrite \eqref{e:XD} as
\begin{eqnarray}\label{e:reXD}
\tl{X}_{k+1}^{\eta}
&=& \tl{X}_{k}^{\eta}+g(\tl{X}_{k}^{\eta}) \eta + \sigma (B_{(k+1){\eta}} -B_{k\eta}), 
\end{eqnarray}
where $k\in \mathbb{N}_0$ and $\tl{X}^{\eta}_{0}$ is the initial value.

\begin{lemma}\label{lem:GePe}
For $\tl{\mathcal{P}}_{\eta}$ before and $f\in \mathcal{B}_b(\R^d,\R)$, one has
\begin{eqnarray*}
\| \nabla \tl{\mathcal{P}}_{\eta} f \|_{\infty} &\leq& \| f \|_{\infty} (1+C_{\rm op}\eta) \| \sigma^{-1} \|^2_{\rm op} \| \sigma \|_{\rm op}\eta^{-\frac{1}{2}}d^{\frac{1}{2}},
\end{eqnarray*}
which implies $(\tl{X}_k^{\eta})_{k\in \mathbb{N}_0}$ is strong Feller. Furthermore, $(\tl{X}_{k}^{\eta})_{k\in \mathbb{N}_0}$ is irreducible.
\end{lemma}

\begin{lemma}\label{lem:Xgesm}
For $(X^{x}_t)_{t\geq 0}$ and $(\tl{X}_{k}^{\eta,x})_{k\in \mathbb{N}_0}$ in \eqref{hSDEg} and \eqref{e:reXD} respectively and integers $\ell\geq 1$, there exists some positive constant $\tl{C}_{\ell}$ depending on $\ell$ not on $\eta$ such that for $0\leq s <1$,
\begin{eqnarray*}
\E |X^{x}_s-x|^{2\ell}  \ \leq \ \tl{C}_{\ell}(1+V^{\ell}(x))s^{\ell}
{\rm \ \ \ and \ \ \ }
\E |X^{x}_{\eta} - \tl{X}^{\eta,x}_{1}|^{2\ell}  \ \leq \  \tl{C}_{\ell} (1+V^{\ell}(x))\eta^{3\ell},
\end{eqnarray*}
where $V$ is in \eqref{e:Lypfun}.
\end{lemma}

\begin{lemma}\label{lem:PXeD}
For $(\tl{X}^{\eta,x}_{k})_{k\in \mathbb{N}_0}$ in \eqref{e:reXD}, $V$ in \eqref{e:Lypfun} and integers $\ell\geq 1$, there exist some constants $\check{\gamma}_{\ell} \in(0,1)$ and $\check{K}_{\ell} \in[0,\infty)$, both depending on $\ell$ but not on $\eta$ such that
\begin{eqnarray*}
\tilde{\mathcal{P}}_{\eta} V^{\ell}(x) &\leq& \check{\gamma}_{\ell} V^{\ell}(x)+\check{K}_{\ell},
\end{eqnarray*}
where $\check{\gamma}_{\ell}=e^{-c_1 \eta} + \tl{C}_{\ell} \eta^{\frac{3}{2}}$ and $\check{K}_{\ell}= \frac{\breve{c}_{\ell}}{c_1}(1-e^{-c_1\eta}) + \tl{C}_{\ell} \eta^{\frac{3}{2}}$ with $c_1$ and $\breve{c}_{\ell}$ in Lemma \ref{lem:AV2}.
\end{lemma}

\begin{proof}[Proof of Proposition \ref{p:GeneralErgodicEM}]
The process $(\tl{X}_k^{\eta})_{k\in \mathbb{N}_0}$ is strong Feller and irreducible from Lemma \ref{lem:GePe}. Then $(\tl{X}_k^{\eta})_{k\in \mathbb{N}_0}$ has at most one invariant measure from \cite[Theorem 1.4]{PZ1}. Combining Feller property in Lemma \ref{lem:GePe} and Lyapunov condition in Lemma \ref{lem:PXeD}, we know $(\tl{X}_k^{\eta})_{k\in \mathbb{N}_0}$ is ergodic with unique invariant measure $\tl \mu_{\eta}$ from \cite[Theorem 4.5]{MT2}. In addition, we prove the exponential ergodicity from \cite{DFMS1, TT1}.

For any $n\in \mathbb{N}_0$ and $x\in \R^d$, let
\begin{eqnarray*}
V_{n}(x) \ = \ e^{\frac{c_1}{8} n\eta}(1+V(x)), \quad
r(n) \  = \  \frac{c_1}{8}\eta e^{\frac{c_1}{8}n\eta}, \quad
\Psi(x) \  =  \ 1+V(x),
\end{eqnarray*}
and
\begin{eqnarray*}
\mathscr{C}\  =  \ \{x: V(x) \  \leq  \ \frac{8 e^{\frac{c_1}{8}\eta}}{c_1 \eta}(1+\check{K}_{1}-\check{\gamma}_{1})-1 \},
\quad
b\  =  \ \frac{8e^{\frac{c_1}{8}\eta}} {c_1\eta}(1+\check{K}_{1}-\check{\gamma}_{1}),
\end{eqnarray*}
where $\check{\gamma}_{1}=e^{-c_1 \eta} + \tl{C}_{1} \eta^{\frac{3}{2}}$, $\check{K}_{1}= \frac{\breve{c}_{1}}{c_1}(1-e^{-c_1\eta}) + \tl{C}_{1} \eta^{\frac{3}{2}}$ with $c_1$ and $\breve{c}_{1}$ in \eqref{e:AV}, $V$ is  in \eqref{e:Lypfun} and the set $\mathscr{C}$ is compact from \eqref{e:BV}. It follows from Lemma \ref{lem:PXeD} that
\begin{eqnarray*}
&& \tilde{\mathcal{P}}_{\eta} V_{n+1}(x) + r(n)\Psi(x) \\
&\leq&  \check{\gamma}_{1}e^{\frac{c_1}{8} \eta} V_n(x) +e^{\frac{c_1}{8}(n+1)\eta} (1+\check{K}_{1}-\check{\gamma}_{1})
+\frac{c_1}{8}\eta V_n(x) \\
&=& V_n(x)+\left( \check{\gamma}_{1}e^{\frac{c_1}{8}\eta} -1+\frac{c_1}{8}\eta \right) e^{\frac{c_1}{8}n\eta}(V(x)+1) +e^{\frac{c_1}{8}(n+1)\eta} (1+\check{K}_{1}-\check{\gamma}_{1})  \\
&=& V_n(x)+ \frac{c_1}{8} \eta e^{\frac{c_1}{8}n\eta}  \left( \frac{ \check{\gamma}_{1}e^{\frac{c_1}{8}\eta}-1 +\frac{c_1}{8}\eta}{ \frac{c_1}{8} \eta} (V(x)+1) + \frac{e^{\frac{c_1}{8}\eta}}{\frac{c_1}{8}\eta} (1+\check{K}_{1}-\check{\gamma}_{1}) \right) \\
&\leq& V_n(x) + b r(n)1_{\mathscr{C}}(x),
\end{eqnarray*}
where the last inequality holds from that
$\check{\gamma}_{1}e^{\frac{c_1}{8}\eta} -1+\frac{c_1}{8}\eta\leq -\frac{1}{8} c_1 \eta$ for small enough $\eta>0$.

We claim that the compact set $\mathscr{C}$ is petite.  It follows from \cite[Theorem 2.1]{TT1} or \cite[Theorem 1.1]{DFMS1} that
\begin{eqnarray*}
\lim_{n\to \infty} r(n) \|\tilde{\mathcal{P}}_{\eta}^n(x,\cdot) -\tl{\mu}_{\eta} \|_{\Psi} \ = \ 0,
\end{eqnarray*}
where $\|\tilde{\mathcal{P}}_{\eta}^n(x,\cdot) -\tl{\mu}_{\eta} \|_{\Psi}= \sup_{|h|\leq \Psi} | \tilde{\mathcal{P}}_{\eta}^n h(x)-\tl{\mu}_{\eta}(h)|$,
combining \eqref{e:BV} with $h(x)\leq C(1+V(x))= C \Psi(x)$ for all $x\in \R^d$ and $h\in {\rm Lip}_0(1)$ and some constant $C\geq 1$, one has
\begin{eqnarray*}
d_W( (\tl{\mathcal{P}}_{\eta}^k)^* \nu, \tl{\mu}_{\eta})
\ &\leq& \ C\eta^{-1} e^{-c k\eta}, \\
\| (\tl{\mathcal{P}}_{\eta}^k)^* \nu- \tl{\mu}_{\eta} \|_{\rm TV}
\ &\leq& \ C\eta^{-1} e^{-c k\eta}.
\end{eqnarray*}
It follows from Lemma \ref{lem:PXeD} that
\begin{eqnarray*}
\int_{\R^d} \tilde{\mathcal{P}}_{\eta} V^{\ell}(x) \tl{\mu}_{\eta}(\dif x)
\ &\leq& \ \int_{\R^d} \check{\gamma}_{\ell} V^{\ell}(x) \tl{\mu}_{\eta}(\dif x) +\check{K}_{\ell},
\end{eqnarray*}
such that
\begin{eqnarray*}
\tl{\mu}_{\eta}(V^{\ell})
\ &\leq& \  \check{\gamma}_{\ell} \tl{\mu}_{\eta}(V^{\ell}) +\check{K}_{\ell},
\end{eqnarray*}
that is,
\begin{eqnarray*}
\tl{\mu}_{\eta}(V^{\ell})
\ &\leq& \ \frac{\check{K}_{\ell}}{1-\check{\gamma}_{\ell}}
\  =  \ \frac{\frac{\breve{c}_{\ell}}{c_1}(1-e^{-c_1\eta}) + \tl{C}_{\ell}  \eta^{\frac{3}{2}}}{1-e^{-c_1 \eta}- \tl{C}_{\ell}\eta^{\frac{3}{2}}}
\  \leq  \  \frac{2\frac{\breve{c}_{\ell}}{c_1} c_1 \eta }{\frac{1}{2}c_1 \eta}
\  = \  \frac{4\breve{c}_{\ell}}{c_1},
\end{eqnarray*}
where the last inequality holds from Taylor expansion for $e^{-c_1 \eta}$ with small $\eta>0$, which implies the desired inequality from the relationship between $V$ and $|\cdot|^2$ in \eqref{e:BV}.

To show the compact set $\mathscr{C}$ is petite.  It suffices to show that
\begin{eqnarray}\label{e:petite}
p(\eta,x,z) \ \geq \ c \nu(z) \quad \forall x\in \mathscr{C},
\end{eqnarray}
where $p(\eta,x,z)$ is the density of $\tl{X}_1^{\eta,x}$, $c$ is some positive constant and $\nu$ is a probability measure from \cite[p. 778]{TT1}. Since
\begin{eqnarray}\label{e:pe}
&& p(\eta,x,z)  \\
&=& ((2\pi)^d \eta^d {\rm det}(\sigma\sigma^{\prime}) )^{-\frac{1}{2}} \exp\left(-(z-x-\eta g(x))^{\prime} \frac{(\sigma\sigma^{\prime})^{-1}} {2\eta}(z-x-\eta g(x)) \right)  \nonumber  \\
&\geq& ((2\pi)^d \eta^d \lambda_M^d)^{-\frac{1}{2}} \exp \left(-\frac{\lambda_m^{-1}} {2\eta}(2|z|^2 + 4|x|^2 + 8\tl{C}_{\rm op}^2 \eta^2(1+|x|^2)) \right)  \nonumber \\
&=& \left( (2\pi)^d \eta^d (\frac{1}{2}\lambda_m)^d \right)^{-\frac{1}{2}} \exp\left(-\frac{|z|^2} {\lambda_m\eta}\right)
\left(\frac{2\lambda_M}{\lambda_m} \right)^{-\frac{d}{2}} \exp\left(-\frac{\lambda_m^{-1}} {2\eta}(4|x|^2 + 8\tl{C}_{\rm op}^2 \eta^2(1+|x|^2)) \right), \nonumber
\end{eqnarray}
where $\lambda_M$ and $\lambda_m$ are maximum and minimum eigenvalues of matrix $\sigma \sigma^{\prime}$ respectively and
$$
|z-x-\eta g(x)|^2 \ \leq \  2|z|^2 + 4|x|^2 + 8\tl{C}_{\rm op}^2 \eta^2(1+|x|^2) \quad  \forall x, z \in \R^d.
$$
Thus, \eqref{e:petite} holds by taking
\begin{eqnarray*}
\nu(z) &=& \left((2\pi)^d \eta^d (\frac{1}{2}\lambda_m)^d \right)^{-\frac{1}{2}} \exp\left(-\frac{|z|^2} {\lambda_m\eta}\right),
\end{eqnarray*}
and
\begin{eqnarray*}
c&=& \inf_{ c\in \mathscr{C}}  \left\{ \left( \frac{2\lambda_M}{\lambda_m} \right)^{-\frac{d}{2}} \exp\left(-\frac{\lambda_m^{-1}} {2\eta}(4|x|^2 + 8\tl{C}_{\rm op}^2 \eta^2(1+|x|^2))\right) \right \}  >0
\end{eqnarray*}
for compact set $\mathscr{C}$. The proof is complete.
\end{proof}

\begin{proof}[Proof of Lemma \ref{lem:AV2}]
(i) Recall that $\mathcal{A} V(x)\leq -c_1V(x)+\breve{c}_1$ for all $x\in \R^d$  in \eqref{e:AV} and the function $V$ in \eqref{e:Lypfun}. 
Combining \eqref{e:BV}, \eqref{e:NV} and using the Young's inequality, then there exists some positive constant $\breve{c}_{\ell}$ such that for all $x\in \R^d$
\begin{eqnarray*}
\mathcal{A}V^{\ell}(x)
&=& \ell V^{\ell-1}(x) \mathcal{A}V(x) + \frac{\ell(\ell-1)}{2}V^{\ell-2}(x) \langle \nabla V(x)(\nabla V(x))^{\prime}, \sigma \sigma^{\prime} \rangle_{\rm HS} \\
&\leq & -c_1 \ell V^{\ell}(x)+\breve{c}_1 \ell V^{\ell-1}(x) + \frac{\ell(\ell-1)}{2}V^{\ell-2}(x) \langle \nabla V(x)(\nabla V(x))^{\prime}, \sigma \sigma^{\prime} \rangle_{\rm HS} \\
&\leq&  -c_1 V^{\ell}(x) + \breve{c}_{\ell}.
\end{eqnarray*}
By using It\^{o}'s formula, we know for all $t\geq 0$
\begin{eqnarray*}
\mathbb{E} V^{\ell}(X_t^{x})
&=& V^{\ell}(x)+\int_0^t\mathbb{E} \mathcal{A}V^{\ell}(X_s^{x})  \dif s
\ \leq \  V^{\ell}(x)+\int_0^t ( -c_1\mathbb{E} V^{\ell}(X_s^{x})+\breve{c}_{\ell}) \dif s,
\end{eqnarray*}
it implies that (\cite[proof of Lemma 7.2]{Gur1})
\begin{eqnarray*}
\mathbb{E} V^{\ell}(X_t^{x})
&\leq& e^{-c_1 t }V^{\ell}(x)+\frac{\breve{c}_{\ell}(1-e^{-c_1t})}{c_1}.
\quad \forall t\geq 0,
\end{eqnarray*}

Let $\chi: [0,\infty) \to [0,1]$ be a continuous function such that $\chi(r)=1$ for $0\leq r \leq 1$ and $\chi(r)=0$ for $r\geq 2$. Let $L>0$ be a large number. It follows from \eqref{e:Vm} that
\begin{eqnarray*}
\mathbb{E} \left[ V^{\ell}(X_t^{x}) \chi\left(\frac{|X_t^{x}|}{L}\right) \right]
&\leq& e^{-c_1 t }V^{\ell}(x)+ \frac{\breve{c}_{\ell} (1-e^{-c_1 t})}{c_1}.
\end{eqnarray*}

Let $t\to \infty$, we have
\begin{eqnarray*}
\int_{\R^d} V^{\ell}(x) \chi\left(\frac{|x|}{L}\right) \mu(\dif x)
&\leq&  \frac{\breve{c}_{\ell}}{c_1}.
\end{eqnarray*}

Let $L \to \infty$, we have
\begin{eqnarray*}
\int_{\R^d} V^{\ell}(x) \mu(\dif x) &\leq&  \frac{\breve{c}_{\ell}}{c_1}.
\end{eqnarray*}
Combining with \eqref{e:BV}, we can get the inequality $\mu(|\cdot|^{2\ell})\leq C$ and $C$ depends on $\ell$.

%

(ii) With similar calculations for \cite[proof of Theorem 3]{DG1} and combining \eqref{e:Vm}, then there exist positive constants $c$ and $C$ such that for any positive integers $\ell$ 
\begin{eqnarray*}
\| P^*_{t}\nu-\mu \|_{\rm{TV}, \rm{V}^{\ell} }
&\leq& C (1+\nu(V^{\ell}))e^{-c t}.
\end{eqnarray*}
Thus, the exponential ergodicity for $(X_t)_{t\geq 0}$ in Wasserstein-1 distance holds from \eqref{e:dWandTV} and \eqref{e:BV} by taking $\ell=1$. Furthermore, one has the following inequality $\| P_{t}^* \nu-\mu \|_{ \rm{TV}} \leq   \| P_{t}^* \nu-\mu \|_{ \rm{TV}, \rm{V} }$ under $V(x)\geq 0$ for all $x\in \R^d$. The proof is complete.
\end{proof}

\begin{proof}[Proof of Lemma \ref{lem:GePe}]
(i) Since $\tl{X}_{1}^{\eta,x}$ has the same law as $\mathcal{N}(x+g(x)\eta, \eta\sigma\sigma^{\prime})$ with the density function $p(\eta,x,z)$ in \eqref{e:pe}, one has
\begin{eqnarray*}
\tilde{\mathcal{P}}_{\eta}f(x)
&=& \E f(\tl{X}_{1}^{\eta,x})
= \int_{\R^d} f(z)p(\eta, x, z) \dif z,
\end{eqnarray*}
thus
\begin{eqnarray*}
\nabla \tilde{\mathcal{P}}_{\eta}f(x)
&=& \int_{\R^d} \eta^{-1} f(z)(I+\nabla g(x)\eta)(\sigma\sigma^{\prime})^{-1} (z-x-g(x)\eta) p(\eta,x,z) \dif z.
\end{eqnarray*}
It implies that
\begin{eqnarray*}
|\nabla \tilde{\mathcal{P}}_{\eta}f(x)|
&\leq& \int_{\R^d} \eta^{-1} |f(z)| \|I+\nabla g(x)\eta\|_{\rm op} \| \sigma^{-1} \|^2_{\rm op}|z-x-g(x)\eta| p(\eta,x,z) \dif z \\
&\leq& \| f \|_{\infty} (1+C_{\rm op}\eta) \| \sigma^{-1} \|^2_{\rm op} \| \sigma \|_{\rm op}\eta^{-\frac{1}{2}}d^{\frac{1}{2}},
\end{eqnarray*}
where the second inequality holds from that $\tl{X}_{1}^{\eta,x}-x-g(x)\eta$ has the same law as $\mathcal{N}(0,\eta\sigma\sigma^{\prime})$.

(ii) For any $x,y\in \R^d$ and $r>0$, it follows from \eqref{e:reXD} that
\begin{eqnarray*}
\tl{\mathcal{P}}_{\eta}(x,B(y,r))
&=& \PP(\tl{X}_{1}^{\eta,x} \in B(y,r) ) 
\ = \ \PP(x+g(x)\eta+\sigma B_{\eta} \in B(y,r)) \\
&=&\PP(\sigma B_{\eta} \in B(y-x- g(x)\eta,r))>0,
\end{eqnarray*}
where $B_{\eta}$ has the same law as $ \mathcal{N}(0,\eta I)$. Assuming that $\tilde{\mathcal{P}}_{\eta}^k(x,B(y,r))>0$ for any $x,y\in\R^d$, $r>0$ and some integer $k$, one has
\begin{eqnarray*}
\tilde{\mathcal{P}}_{\eta}^{k+1}(x,B(y,r))
=\int_{\R^d} \tilde{\mathcal{P}}_{\eta}^{k}(x,z) \tilde{\mathcal{P}}_{\eta}(z,B(y,r)) \dif z
>0,
\end{eqnarray*}
thus, the irreducibility holds by induction.
\end{proof}

\begin{proof}[Proof of Lemma \ref{lem:Xgesm}]

(i) Combining with Lemma \ref{lem:AV2} and with similar calculations for \eqref{e:Vm}, one has
\begin{eqnarray}\label{e:V4}
\E V^{\ell}(X^{x}_{\eta})
&\leq& e^{-c_1 \eta} V^{\ell}(x) + \frac{\breve{c}_{\ell}}{c_1}(1-e^{-c_1 \eta}).
\end{eqnarray}

From $(X_t)_{t \geq 0}$ in SDE \eqref{hSDEg}, one has for any $0<s<1$,
\begin{eqnarray*}
X^{x}_s &=& x + \int_0^s  g(X^{x}_u) \dif u +\sigma B_s,
\end{eqnarray*}
by using H\"{o}lder inequality and $|g(x)| \leq \tl{C}_{{\rm op}}(1+|x|)$ for all $x\in \R^d$ with $\tl{C}_{{\rm op}}$ in \eqref{tlCop}, there exists some positive constant $\tl{C}_{\ell}$ depending on $\ell$ such that
\begin{eqnarray*}
|X^{x}_s-x|^{2\ell}
&=& \left| \int_0^s g(X^{x}_u)\dif u+\sigma B_s \right|^{2\ell} 
\  \leq \   \tl{C}_{\ell} s^{2\ell-1} \int_0^s  (1+V^{\ell}(X^{x}_u) ) \dif u + \tl{C}_{\ell}| B_s |^{2\ell},
\end{eqnarray*}
where the last inequality holds from \eqref{e:BV}. Combining with \eqref{e:V4}, we know there exists some positive constant $\tl{C}_{\ell}$ depending on $\ell$ such that
\begin{eqnarray}\label{e:Exe}
\E |X^{x}_s -x |^{2\ell}
&\leq& \tl{C}_{\ell} s^{2\ell-1} \int_0^s  (1+\E V^{\ell}(X^{x}_u) ) \dif u + \tl{C}_{\ell} \E | B_s |^{2\ell} 
\ \leq \  \tl{C}_{\ell} s^{\ell}(1+V^{\ell}(x)),
\end{eqnarray}
where the second inequality holds from \eqref{e:V4} and the last inequality holds for $0<s<1$.

(ii) From  $(X_t)_{t\geq 0}$ and $(\tl{X}^{\eta}_{k})_{k\in \mathbb{N}_0}$ in \eqref{hSDEg} and \eqref{e:reXD} respectively, one has
\begin{eqnarray*}
X^{x}_{\eta}  - \tl{X}^{\eta,x}_{1}
&=& \int_0^{\eta} ( g(X^{x}_s) - g(x) ) \dif s,
\end{eqnarray*}
by using H\"{o}lder inequality and $\| \nabla g(x) \|_{{\rm op}}\leq C_{{\rm op}}$ for all $x\in \R^d$ with $C_{{\rm op}}$ in \eqref{Cop}, there exists some positive constant $\tl{C}_{\ell}$ depending on $\ell$ not on $\eta$ such that
\begin{eqnarray*}
|X^{x}_{\eta}  - \tl{X}^{\eta,x}_{1} |^{2\ell}
&\leq& \left| \int_0^{\eta} ( g(X^{x}_s) - g(x) ) \dif s \right|^{2\ell}
\ \leq \ \tl{C}_{\ell} \eta^{2\ell-1} \int_0^{\eta} |X^{x}_s - x|^{2\ell} \dif s .
\end{eqnarray*}
Combining this with \eqref{e:Exe}, there exists some positive constant $\tl{C}_{\ell}$ depending on $\ell$ not on $\eta$ such that
\begin{eqnarray*} 
\E |X^{x}_{\eta} - \tl{X}^{\eta,x}_{1} |^{2\ell}
&\leq& \tl{C}_{\ell} \eta^{2\ell-1} \int_0^{\eta} s^{\ell}(1+V^{\ell}(x)) \dif s
\ \leq \ \tl{C}_{\ell} (1+V^{\ell}(x))\eta^{3\ell}.
\end{eqnarray*}
The proof is complete.
\end{proof}

\begin{proof}[Proof of Lemma \ref{lem:PXeD}]

We use the method from \cite[Theorem 7.2]{MSH1} to get the desired inequality. For $V$ in \eqref{e:Lypfun}, one has
\begin{eqnarray*}
\E V^{\ell}(\tl{X}^{\eta,x}_{1})
&\leq& \E V^{\ell}(X^{x}_{\eta})+\E| V^{\ell}(\tl{X}^{\eta,x}_{1})-V^{\ell}(X^{x}_{\eta})|.
\end{eqnarray*}
Let $\tl{C}_{\ell}$ be some constants depending on $\ell$ but not on $\eta$, whose values may vary from line to line.

We claim that for small $\eta \in (0,e^{-1})$,
\begin{eqnarray}\label{e:V-V4}
\E |V^{\ell}(\tl{X}^{\eta,x}_{1}) -V^{\ell}(X^{x}_{\eta})|
&\leq& \tl{C}_{\ell} (1+V^{\ell}(x))\eta^{\frac{3}{2}}.
\end{eqnarray}
Combining this with \eqref{e:V4}, one has
\begin{eqnarray*}
\E V^{\ell}(\tl{X}^{\eta,x}_{1})
&\leq& \E V^{\ell}(X^{x}_{\eta}) + \E | V^{\ell}(\tl{X}^{\eta,x}_{1}) -V^{\ell}(X^{x}_{\eta})| \\
&\leq& e^{-c_1 \eta} V^{\ell}(x) +  \frac{\breve{c}_{\ell}}{c_1}(1-e^{-c_1\eta}) + \tl{C}_{\ell} \eta^{\frac{3}{2}} (1+V^{\ell}(x))  \\
&\leq& (e^{-c_1 \eta} + \tl{C}_{\ell} \eta^{\frac{3}{2}} ) V^{\ell}(x) + \frac{\breve{c}_{\ell}}{c_1}(1-e^{-c_1\eta}) + \tl{C}_{\ell} \eta^{\frac{3}{2}},
\end{eqnarray*}
which implies
\begin{eqnarray*}
\tilde{\mathcal{P}}_{\eta} V^{\ell}(x)
&\leq& \check{\gamma}_{\ell} V^{\ell}(x)+\check{K}_{\ell},
\end{eqnarray*}
with $\check{\gamma}_{\ell}=e^{-c_1 \eta} + \tl{C}_{\ell} \eta^{\frac{3}{2}}$ and $\check{K}_{\ell}= \frac{\breve{c}_{\ell}}{c_1}(1-e^{-c_1\eta}) + \tl{C}_{\ell} \eta^{\frac{3}{2}}$ for small enough $\eta\in (0,e^{-1})$.

It remains to show the claim \eqref{e:V-V4} holds.
Since
\begin{eqnarray*}
V^{\ell}(\tl{X}^{\eta,x}_{1})-V^{\ell}(X^{x}_{\eta}) &=& \ell \int_0^1 (\tl{X}^{\eta,x}_{1}-X^{x}_{\eta})^{\prime} \nabla V( X^{x}_{\eta} + r(\tl{X}^{\eta,x}_{1}-X^{x}_{\eta})  ) V^{\ell-1}( X^{x}_{\eta} + r(\tl{X}^{\eta,x}_{1}-X^{x}_{\eta})  )   \dif r,
\end{eqnarray*}
then by \eqref{e:BV}
\begin{eqnarray*}
|V^{\ell}(\tl{X}^{\eta,x}_{1})-V^{\ell}(X^{x}_{\eta})|
&\leq& \tl{C}_{\ell} \int_0^1 [1+|X^{x}_{\eta} + r(\tl{X}^{\eta,x}_{1}-X^{x}_{\eta})|]^{2\ell-1}  |\tl{X}^{\eta,x}_{1}-X^{x}_{\eta}| \dif r  \\
& \leq & \tl{C}_{\ell} (1+|X^{x}_{\eta}|^{2\ell-1})| \tl{X}^{\eta,x}_{1}-X^{x}_{\eta}| +\tl{C}_{\ell} |\tl{X}^{\eta,x}_{1}-X^{x}_{\eta}|^{2\ell}.
\end{eqnarray*}
By H\"{o}lder inequality and \eqref{e:BV}, one has
\begin{eqnarray*} 
\E |V^{\ell}(\tl{X}^{\eta,x}_{1}) -V^{\ell}(X^{x}_{\eta})|  
&\leq& \tl{C}_{\ell}(\E [1+|X^{x}_{\eta}|^{2\ell-1}] ^{\frac{2\ell}{2\ell-1}})^{\frac{2\ell-1}{2\ell}} [\E |\tl{X}^{\eta,x}_{1}-X^{x}_{\eta}|^{2\ell}] ^{\frac{1}{2\ell}}+ \tl{C}_{\ell} \E |\tl{X}^{\eta,x}_{1}-X^{x}_{\eta}| ^{2\ell} \\
& \leq& \tl{C}_{\ell}(\E [1+V^{\ell}(X^{x}_{\eta})]) ^{\frac{2\ell-1}{2\ell}} [\E |\tl{X}^{\eta,x}_{1}-X^{x}_{\eta}|^{2\ell}] ^{\frac{1}{2\ell}}
+ \tl{C}_{\ell} \E |\tl{X}^{\eta,x}_{1}-X^{x}_{\eta}|^{2\ell}.
\end{eqnarray*}
Combining this with \eqref{e:V4} and Lemma \ref{lem:Xgesm},  for small $\eta \in (0,e^{-1})$ we have,
\begin{eqnarray*}
\E |V^{\ell}(\tl{X}^{\eta,x}_{1})-V^{\ell}(X^{x}_{\eta})|
&\leq& \tl{C}_{\ell} \left(1+ e^{-c_1 \eta}V^{\ell}(x) +  \frac{\breve{c}_{\ell}}{c_1} \right)^{\frac{2\ell-1}{2\ell}} \eta^{\frac{3}{2}}(1+V^{\ell}(x))^{\frac{1}{2\ell}} + \tl{C}_{\ell} (1+V^{\ell}(x))\eta^{3\ell} \\
&\leq&\tl{C}_{\ell} (1+V^{\ell}(x)) \eta^{\frac{3}{2}}.
\end{eqnarray*}
The proof is complete.
\end{proof}

\section{Proof of auxiliary Lemmas in Section \ref{Malliavin-Stein} } \label{sec:AAS}

\begin{proof}[Proof of Lemma \ref{lem:XXem}]
(i) By using It\^o's formula, for any integers $m \geq 2$, we have
\begin{eqnarray*}
\E |X^{x}_{t}|^{m}&=&|x|^{m}+m \mathbb{E} \int_{0}^{t} |X^{x}_{s}|^{m-2} (X^{x}_{s})' g(X^{x}_{s}) \dif s \\
&&+\frac m2 \mathbb{E} \int_{0}^{t} |X^{x}_{s}|^{m-4}\left[(m-2)|\sigma' X^{x}_{s}|^{2}+ { \rm tr} (\sigma \sigma') |X^{x}_{s}|^{2} \right] \dif s.
\end{eqnarray*}
By the bound $|g(x)| \leq \tl{C}_{\rm op}(1+|x|)$ for all $x\in \R^d$ with $\tl{C}_{{\rm op}}$ in \eqref{tlCop}, we further get
\begin{eqnarray*}
\E |X^{x}_{t}|^{m}& \leq &|x|^{m}+\tilde{C}_{\rm op} \left(  m\int_{0}^{t} \E |X^{x}_{s}|^{m} \dif s+  m\int_{0}^{t} \E |X^{x}_{s}|^{m-1} \dif s+ m^2\int_{0}^{t} \E |X^{x}_{s}|^{m-2} \dif s\right)  \nonumber \\
&\le & |x|^{m}+2m^2 \tilde{C}_{\rm op} \left(  \int_{0}^{t} \E |X^{x}_{s}|^{m} \dif s+t \right),
\end{eqnarray*}
where the second inequality is by Young's inequalities, that is, $b^{m-1} \le \frac{m-1}{m} b^{m}+\frac 1m$ and $b^{m-2} \le \frac{m-2}{m} b^{m}+\frac 2m$. Thus, we have
\begin{eqnarray*}
\E |X^{x}_t|^{m}  \ \le \  e^{C_m t}(|x|^{m}+1),
\end{eqnarray*}
where $C_m=2m^2 \tl{C}_{\rm op}$. The moment estimates $\E |X^{\e,x}_t|^{m}$ can be obtainded similarly. Thus, inequality \eqref{e:XXem} holds.

(ii) Consider $X^{\e,x}_{t}-X_{t}^{x}$, which satisfies the following equation
\begin{eqnarray*}
\frac{\dif}{\dif t} \left(X^{\e,x}_{t}-X_{t}^{x}\right)&=&g_\e (X^{\e,x}_t)-g(X^{x}_t) \\
&=&g_\e (X^{\e,x}_t)-g_{\e}(X^{x}_t)+g_{\e}(X^{x}_t)-g(X^{x}_{t}), \\
&=&\nabla g_{\e}(\theta_{t})  \left(X^{\e,x}_{t}-X_{t}^{x}\right)+g_{\e}(X^{x}_t)-g(X^{x}_{t}),
\end{eqnarray*}
where $\theta_{t}$ is between $X^{x}_t$ and $X^{x,\e}_t$. The above equation can be solved by
\begin{eqnarray*}
X^{\e,x}_{t}-X_{t}^{x}&=&\int_{0}^{t} \exp\left(\int_{s}^{t} \nabla g_{\e}(\theta_{r}) \dif r \right) (g_{\e}(X^{x}_s)-g(X^{x}_{s})) \dif s.
\end{eqnarray*}
Since $\|\nabla g_{\e}(x)\|_{ {\rm op} } \le C_{{\rm op}}$ for all $x \in \R^d$, the  relation $|g_{\e}(x)-g(x)| \le C_{{\rm op}}\e$ for all $x \in \R^d$ immediately gives us \eqref{e:XeCon-1}.
\end{proof}

\begin{proof}[Proof of Lemma \ref{l:XeCon}]
Denote the event
$$N=\left\{\int_{0}^{\infty} \|\nabla g_{\e}(X^{\e,x}_{s})\|_{ {\rm op} }1_{\{{\rm e}^{\prime}  X^{x}_{s}=0\}} \dif s\ne 0\right\},$$
and we claim that
\Be  \label{e:PN=0}
\PP(N) \ = \ 0.
\Ee
Indeed, for any $T>0$, by $\|\nabla g_{\e}(x)\|_{{\rm op}} \le C_{{\rm op}}$ for all $x \in \R^d$ and $\e$,  we have
\begin{eqnarray}\label{e:ET}
\E \int_{0}^{T} \|\nabla g_{\e}(X^{\e,x}_{s})\|_{ {\rm op} }1_{\{{\rm e}^{\prime}  X_{s}^x=0\}} \dif s
&=& \int_{0}^{T} \E[\|\nabla g_{\e}(X^{\e,x}_{s})\|_{{\rm op}}1_{\{{\rm e}^{\prime}  X_{s}^x=0\}}] \dif s  \nonumber \\
& \leq &   \int_{0}^{T}  C_{{\rm op}}  \E 1_{\{{\rm e}^{\prime}  X_{s}^x=0\}} \dif s  \nonumber \\
&=& 0,
\end{eqnarray}
where the last equality is by Proposition \ref{lem:occupation} and the fact that for any small $\e>0$,
\begin{eqnarray*}
 \E\int_0^T  1_{\{{\rm e}^{\prime}  X_{s}^x=0\}} \dif s 
& \leq &  \E L_T^{\e,x} 
 \ \leq \ C\e e^{\frac{C_2}{2} T} (1+|x|)(1+T),
\end{eqnarray*}
while $L_t^{\e,x} = \int_0^t [-\frac{1}{\e^2} ({\rm e}^{\prime} X_s^x)^2 +1 ] 1_{ \{ |{\rm e}^{\prime} X_s^x| \leq \e \} } \dif s$. The above inequality holds for any $\e>0$, we know $  \E\int_0^T  1_{\{{\rm e}^{\prime}  X_{s}^x=0\}} \dif s = 0$.

Since \eqref{e:ET} holds for all $T>0$, we see that
\begin{eqnarray*}
\E \int_{0}^{\infty} \|\nabla g_{\e}(X^{\e,x}_{s})\|_{ {\rm op} }1_{\{{\rm e}^{\prime}  X^{x}_{s}=0\}} \dif s
& = & 0,
\end{eqnarray*}
hence \eqref{e:PN=0} holds.

Recall the definition of $J^{\e,x}_{s,t}$ and define
$$\hat J^{\e,x}_{s,t}:=\exp \left(\int_s^t \nabla g_{\e} (X_{r}^{\e,x}) 1_{\{{\rm e}^{\prime}  X^{x}_{r} \ne 0\}}\dif r \right).$$
It is easy to verify that
\Be  \label{e:JstConN}
\lim_{\e \rightarrow 0} \hat J^{\e,x}_{s,t}=J^{x}_{s,t}, \ \ \ \ \ \ 0 \le s \leq t<\infty.
\Ee
For any $\omega \notin N$, we know $\int_{0}^{\infty} \nabla g_{\e}(X^{\e,x}_{s})1_{\{{\rm e}^{\prime}  X^{x}_{s}=0\}} \dif s=0$ and thus
$$\exp\left(\int_{s}^{t} \nabla g_{\e}(X^{\e,x}_{r})1_{\{{\rm e}^{\prime}  X^{x}_{r}=0\}} \dif r \right)=I, \ \ \ \ 0 \le s \leq t<\infty.$$
Since $I$ commutes with any matrix, for all $\omega \notin N$, we get
$$\hat J^{\e,x}_{s,t}=\hat J^{\e,x}_{s,t} \exp \left(\int_{s}^{t} \nabla g_{\e}(X^{\e,x}_{r})1_{\{{\rm e}^{\prime}  X^{x}_{r}=0\}} \dif r \right)=J^{\e,x}_{s,t}, \ \ \ \  \ 0 \le s \leq t<\infty.$$
This, combining with \eqref{e:JstConN}, implies that for all $\omega \notin N$,
$$\lim_{\e \rightarrow 0}  J^{\e,x}_{s,t}=J^{x}_{s,t}, \ \ \ \ \ \ 0 \le s \leq t<\infty.$$
Note that $ J^{\e,x}_{s,t}$ and $J^{x}_{s,t}$ are matrices, the above pointwise convergence implies the convergence in operator.
\end{proof}

\begin{proof}[Proof of Lemma \ref{lem:EIm}]
By using Burkholder-Davis-Gundy inequality, we have
\begin{eqnarray*}
\E \left|\mathcal{I}_{u}^{\e,x}(t)\right|^m
&\leq &\frac{C}{t^m}\E \left(\int_0^t  |\sigma^{-1}J^{\e,x}_{r} u|^2 \dif r \right)^{m/2}
\le  \frac{ C }{t^m} \E \left(\int_0^t \|\sigma^{-1}\|^{2}_{ {\rm op} } \|J^{\e,x}_{r}\|^{2}_{ {\rm op} }  |u|^{2}\dif r \right)^{m/2},
\end{eqnarray*}
which, together with \eqref{e:JJe}, immediately gives \eqref{e:IuexEst}.

For the second relation, by Burkholder-Davis-Gundy inequality, we have
\begin{eqnarray*}
\E |\mathcal{I}_{u}^{\e,x}(t)- \mathcal{I}_{u}^{x}(t)|^m
&=& \E \left|\frac1t\int_0^t \langle \sigma^{-1}(J^{\e,x}_{r}-J^{x}_{r}) u, \dif B_r\rangle \right|^m \\
&\leq &\frac{ C }{t^m} \E \left(\int_0^t  |\sigma^{-1}(J^{\e,x}_{r}-J^{x}_{r}) u|^2 \dif r \right)^{m/2} \ \rightarrow \ 0  { \ \ \rm as \ \ } \e \rightarrow 0,
\end{eqnarray*}
where the limit is by dominated convergence theorem (with a notice of Lemma \ref{l:XeCon}).
\end{proof}

\begin{proof}[Proof of Lemma \ref{hLef2}]
If $\psi\in \mathcal{C}^1(\mathbb{R}^d,\mathbb{R})$, then by \eqref{e:DVNu} and the Bismut formula \eqref{e:BisFor}, we have
\begin{eqnarray*}
\nabla_{u} \mathbb{E}[\psi(X^{\e,x}_t)] &=& \mathbb{E}[\nabla \psi(X^{\e,x}_t) \nabla_{u} X^{\e,x}_t  ]  
= \mathbb{E}[\nabla \psi(X^{\e,x}_t) D_{\mathbb{V}}X^{\e,x}_t]  \\
&=&  \mathbb{E}[D_{\mathbb{V}} \psi(X^{\e,x}_t)] 
= \mathbb{E} [\psi(X^{\e,x}_t) \mathcal{I}_{u}^{\e,x}(t)],
\end{eqnarray*}
where $\mathbb{V}$ is the direction of Malliavin derivative.

Because the operator $\nabla$ is closed, (\cite[Theorem 2.2.6]{PJR1}) and by the well known property of closed operators \cite[Proposition 2.1.4]{PJR1}, as long as it is shown that
\begin{eqnarray}  \label{e:EPhi}
\lim_{\e \rightarrow 0} \mathbb{E}[\psi(X^{\e,x}_t)] \ = \ \mathbb{E}[\psi(X^{x}_t)],
\  \ \lim_{\e \rightarrow 0} \mathbb{E} [\psi(X^{\e,x}_t) \mathcal{I}_{u}^{\e,x}(t)]\ =\ \mathbb{E} [\psi(X^{x}_t) \mathcal{I}_{u}^{x}(t)],
\end{eqnarray}
then we know that $\nabla_u\E[\psi(X^x_t)]$ exists and has its value as $\mathbb{E} [\psi(X^{x}_t) \mathcal{I}_{u}^{x}(t)]$. Hence, the first relation is proved.

Before proving \eqref{e:EPhi}, let us show that for all $m \ge 1$,
\begin{eqnarray}  \label{e:PhiPowM}
\mathbb{E} |\psi(X^{\e,x}_t)|^{m},  \mathbb{E} |\psi(X^{x}_t)|^{m} & \le & \hat{C}_1 e^{C_m t} (|x|^{m}+1),
\end{eqnarray}
where $\hat{C}_1$ does not depend on $\e$, $t$ and $x$. (Without loss of generality, we can take $\hat{C}_1$ depending on $m$, $\| \nabla \psi \|_{\infty}$ and $\psi(0)$.) Indeed, it is easily seen that
 \begin{eqnarray*}
\mathbb{E} |\psi(X^{\e,x}_t)|^{m}
& \le & \hat{C}_2 \left(|\psi(0)|^{m}+ \| \nabla \psi \|^{m}_{\infty} \E |X^{\e,x}_t|^{m} \right),
 \end{eqnarray*}
where $\hat{C}_2$ only depends on $m$ and this, together with \eqref{e:XXem}, immediately yields the aimed inequality.

For the first limit of \eqref{e:EPhi}, for all $m \ge 1$, by \eqref{e:XeCon-1}  we have
 \begin{eqnarray}
 \mathbb{E} |\psi(X^{\e,x}_t)-\psi(X^{x}_{t})|^{m} & \le & \| \nabla \psi \|^{m}_{\infty}  \E |X^{\e,x}_t-X^{x}_t |^{m}  \ \ \rightarrow \ 0 \textrm{ \ \ \rm as \ \ } \e \to 0.   \label{e:PhiXXe}
 \end{eqnarray}
 For the second one, we have
 \begin{eqnarray*}
&&  |\mathbb{E} [\psi(X^{\e,x}_t) \mathcal{I}_{u}^{\e,x}(t)]-\mathbb{E} [\psi(X^{x}_t) \mathcal{I}_{u}^{x}(t)]|  \\
 &\leq &  \mathbb{E} [|\psi(X^{\e,x}_t)| |\mathcal{I}_{u}^{\e,x}(t)-\mathcal{I}_{u}^{x}(t)|]+\mathbb{E} [|\psi(X^{\e,x}_t)-\psi(X^{x}_t)||\mathcal{I}_{u}^{x}(t)|] \\
 & \rightarrow & 0 { \ \ \rm as  \ \ } \e \rightarrow 0,
 \end{eqnarray*}
where the convergence is by the following argument: applying Cauchy-Schwarz inequality on the two expectations, and using \eqref{e:IuexEst}, \eqref{e:IueCon}, \eqref{e:PhiPowM} and \eqref{e:PhiXXe}.
\end{proof}

\begin{proof}[Proof of Lemma \ref{prop:ST}]
We firstly show that $\int_0^{\infty}[P_t h(x) - \mu(h)] \dif t$ is well defined. Denote $\hat{h} = -h + \mu(h)$. For  any $h \in \mathcal{B}_b(\R^d,\R)$, we know  $h(x) \leq \|h\|_{\infty}( 1+ V(x))$ for all $x\in \R^d$ and from Lemma \ref{lem:AV2} , one has
	\begin{eqnarray*}
	|P_t h(x) - \mu(h) | \ \leq \ \|h\|_{\infty} \| P^*_{t}\delta_x - \mu    \|_{ \rm{TV}, \rm{V} }
	\ \leq \ C \|h\|_{\infty}(1+ V(x)) e^{-c t},
	\end{eqnarray*}
which implies that
	\begin{eqnarray*}
\left| \int_0^{\infty}[P_t h(x)-\mu(h)] \dif t \right|
\ \leq \
C \|h\|_{\infty}(1+V(x)) 
\ < \ \infty.
	\end{eqnarray*}
The reminder is similar to that of \cite[Proposition 6.1]{FSX1}. The proof is complete.
%
%
%
%
%
%
\end{proof}

\begin{proof}[Proof of Lemma \ref{lem:regf}]
(i) It follows from \eqref{e:SE1} that 
	\begin{eqnarray*}
|f(x)|
& \leq &  \int_0^{\infty} |P_t h(x) - \mu(h) |  \dif t
\ \leq \  \|h\|_{\infty} \int_0^{\infty} \| P^*_{t}\delta_x - \mu \|_{ \rm{TV}, \rm{V} }   \dif t.
	\end{eqnarray*}
Then, we know
\begin{eqnarray}\label{e:SE2}
|f(x)| &\leq& C \|h\|_{\infty}(1+V(x)  )
\  \leq \   C \|h\|_{\infty}(1+|x|^2),
\end{eqnarray}
where the last inequality holds from the relationship between $V$ and $|\cdot|^2$ in \eqref{e:BV}.

(ii) Firstly, let  $h \in \mathcal{C}_b^1(\R^d,\R)$, we have $\nabla_u\E[h(X_t^x)]=\E[\nabla_u h(X_t^x)]$ by Lebesgue's dominated convergence theorem. Then we consider the term $ \nabla_u [e^{-\lambda t} \E f(X_t^x)]$. Recall that
	$$
	f(x)=-\int_0^{\infty}P_t[h(x)-\mu(h)]\dif t.
	$$
	By  Lemma \ref{hLef2}, one has
	\begin{eqnarray*}
		\E [\nabla_u f(X_t^x)]
		\ =\ \E[\nabla f(X_t^x) \nabla_u X_t^x]
		\ = \ \E[\nabla f(X_t^x) D_{\mathbb{V}} X_t^x]
		\ = \ \E[D_{\mathbb{V}} f(X_t^x)]
	\ = \ \E[f(X_t^x)\mathcal{I}_{u}^x(t)].
	\end{eqnarray*}

From \eqref{e:SE2}, one has
	\begin{eqnarray}\label{e:f2}
[ \E |f(X_t^x)|^2 ]^{\frac{1}{2}}
\ \leq \ C [ 1+ \E V^2(X_t^x)] ^{\frac{1}{2}}
\ \leq \ C e^{\frac{C_4}{2} t}(1+|x|^2).
	\end{eqnarray}
Combining with the estimate for  $\E | \mathcal{I}_{u}^x(t)|^2$ in \eqref{e:IuexEst} and using H\"{o}lder inequality, one has
	\begin{eqnarray*}
\E|f(X_t^x)\mathcal{I}_{u}^x(t)|
&\leq& [\E f^2 (X_t^x) ]^{\frac{1}{2}} [\E |\mathcal{I}_{u}^x(t) |^2 ]^{\frac{1}{2}}
\ \leq \ C e^{\frac{C_4}{2} t }(1+ |x|^2)  |u| t^{-\frac{1}{2}} e^{C_{{\rm op}}t}
\ < \  \infty,
	\end{eqnarray*}
and	with similar calculations, one has
	\begin{eqnarray*}
		\E|h(X_t^x)\mathcal{I}_{u}^x(t)|
		&\leq&  \| h \|_{\infty}  \E | \mathcal{I}_{u}^x(t) | 		
		\ < \  \infty.
	\end{eqnarray*}
	Here we fix $x$ and $t>0$. Then by Lebesgue's dominated convergence theorem, we have
	\begin{eqnarray*}
		\nabla_u\E[e^{-\lambda t}f(X_t^x)] \ =\ \E[e^{-\lambda t}\nabla_u f(X_t^x)].
	\end{eqnarray*}
	
	It follows from H\"{o}lder inequality that
	\begin{eqnarray}\label{e:Nfx0}
	&& \int_0^{\infty} \left|e^{-\lambda t} (  \lambda  \nabla_{u}\E [ f (X_t^x)  ]  -  \nabla_u\E [   h(X_t^x)] )\right| \dif t  \nonumber \\
	&= & \int_0^{\infty}\left| e^{-\lambda t} (  \lambda  \E [ f (X_t^x)  \mathcal{I}_{u}^x(t) ]  -  \E [ h(X_t^x) \mathcal{I}_u^x(t) ] ) \right|\dif t   \nonumber  \\
	 &\leq&  \int_0^{\infty} e^{-\lambda t}   \lambda  [ \E f^2 (X_t^x) ]^{\frac{1}{2}} [ \E | \mathcal{I}_{u}^x(t)  |^2 ]^{\frac{1}{2}}    \dif t
	+  \int_0^{\infty} e^{-\lambda t}   \| h \|_{\infty}  \E [ |\mathcal{I}_u^x(t)| ]  \dif t.
	\end{eqnarray}

	From estimates for  $\E | \mathcal{I}_{u}^x(t)  |^2$ in \eqref{e:IuexEst}  and  $[ \E |f(X_t^x)|^2 ]^{\frac{1}{2}} $ in \eqref{e:f2}, one has
\begin{eqnarray}\label{e:Nfx1}
\int_0^{\infty} e^{-\lambda t}   \lambda  [ \E f^2 (X_t^x) ]^{\frac{1}{2}} [ \E | \mathcal{I}_{u}^x(t)  |^2 ]^{\frac{1}{2}}  \dif t
&\leq& \int_0^{\infty}  \lambda (1+ |x|^2) e^{\frac{C_4}{2} t } |u| t^{-\frac{1}{2}} e^{ (-\lambda + C_{ {\rm op} })t  } \dif t \nonumber  \\
&\leq&  C (1+ |x|^2) |u|,
	\end{eqnarray}
	where the last inequality holds from taking $\lambda>\frac{C_4}{2} + C_{ \textrm{op} }$.

	From estimate for  $\E | \mathcal{I}_{u}^x(t)  |$ in \eqref{e:IuexEst}, one has
	\begin{eqnarray}\label{e:Nfx2}
\int_0^{\infty} e^{-\lambda t}   \|  h \|_{\infty}  \E [ | \mathcal{I}_u^x(t) | ]    \dif t
&\leq& \int_0^{\infty} e^{-\lambda t}  \|  h \|_{\infty} \frac{C|u|}{ t^{1/2} }  e^{C_{ \textrm{op} } t} \dif t
\ \leq \ C\|h\|_{\infty} |u|,
	\end{eqnarray}
	where the last inequality holds from taking $\lambda> C_{ \textrm{op} }$.

	Since
	\begin{eqnarray*}
		f(x)=  \int_0^{\infty} e^{-\lambda t} P_t[  \lambda f(x) - h(x) + \mu(h) ] \dif t,  \quad  \forall  \lambda>0,
	\end{eqnarray*}
	by Lebesgue's dominated convergence theorem, we have
	\begin{eqnarray}\label{e:nuf}
		\nabla_u f(x)&=&  \int_0^{\infty}  \nabla_u [\lambda e^{-\lambda t} \E f(X_t^x) - e^{-\lambda t} \E h(X_t^x)] \dif t \nonumber  \\
		&=& \int_0^{\infty} e^{-\lambda t} (  \lambda  \nabla_{u}\E [ f (X_t^x)  ]  -  \nabla_u\E [   h(X_t^x)] ) \dif t   \nonumber \\
		&=&\int_0^{\infty} e^{-\lambda t} (  \lambda  \E [ f (X_t^x)  \mathcal{I}_{u}^x(t) ]  -  \E [  h(X_t^x) \mathcal{I}_{u}^x(t) ] ) \dif t.
	\end{eqnarray}
	
	Taking $\lambda= \frac{C_4}{2} + C_{ \textrm{op} } +1$, combining \eqref{e:Nfx0}, \eqref{e:Nfx1} and \eqref{e:Nfx2}, one has
	\begin{eqnarray*}
|\nabla_{u} f(x)|
&\leq& C \| h \|_{\infty} (1+ |x|^2) |u|.
	\end{eqnarray*}

Secondly, we extend $h\in \mathcal{C}_b^1(\R^d,\R)$ to $h\in \mathcal{B}_b(\R^d,\R)$ by using standard approximation (see \cite[pp. 968-969]{FSX1}). Let $h\in \mathcal{B}_b(\R^d,\R)$, and define
	\begin{eqnarray*}
		h_{\delta}(x) = \int_{\R^d} \varphi_{\delta}(y) h(x-y) \dif y, \quad \delta>0,
	\end{eqnarray*}
where $\varphi_{\delta}$ is the density function of the normal distribution $\mathcal{N}(0,\delta^2 I)$. Thus $h_{\delta}$ is smooth, $\| h_{\delta} \|_{\infty}  \leq \| h\|_{\infty}$ and the solution to the Poisson equation \eqref{e:SE} with $h$ replaced by $h_{\delta}$, is
         \begin{eqnarray*}
	f_{\delta}(x)= - \int_0^{\infty} P_t[   h_{\delta}(x) - \mu(h_{\delta}) ] \dif t.
	\end{eqnarray*}

	Denote $\hat{h}_{\delta} = -h_{\delta} + \mu(h_{\delta})$.  Since $h_{\delta} \in \mathcal{C}_b^1(\R^d,\R)$, one has $h_{\delta}(x) \leq \|h_{\delta}\|_{\infty}( 1+ V(x))$ for all $x\in \R^d$. From Lemma \ref{lem:AV2} , one has
	\begin{eqnarray*}
|P_t h_{\delta}(x) - \mu(h_{\delta}) |
&\leq& \|h_{\delta}\|_{\infty} \| P^*_{t}\delta_x - \mu \|_{\rm{TV},\rm{V}}
\ \leq \ C \|h_{\delta}\|_{\infty}(1+V(x)) e^{-ct}.
	\end{eqnarray*}

With similar calculations in the proof for Lemma \ref{prop:ST}, one has
	\begin{eqnarray*}
|f_{\delta}(x)|
&\leq&  C \|h_{\delta}\|_{\infty}(1+ |x|^2 ).
\end{eqnarray*}
Then \eqref{e:SE2}  holds.

	By the dominated convergence theorem, one has
	\begin{eqnarray*}
	\lim_{\delta \to 0} f_{\delta}(x) \ = \ - \int_0^{\infty} P_t[   h(x) - \mu(h) ] \dif t
	\ = \ f(x).
	\end{eqnarray*}
	
From \eqref{e:SE2} and $\| h_{\delta} \|_{\infty} \leq \| h \|_{\infty}$, we know
	\begin{eqnarray*}
| f_{\delta}(x)|
&\leq& C \|h_{\delta}\|_{\infty}( 1 + |x|^2 )
\ \leq \ C \|h\|_{\infty}(1+ |x|^2).
\end{eqnarray*}
Letting $\delta \to 0$, we know
\begin{eqnarray*}
| f(x)|
\ \leq \ C \|h\|_{\infty}( 1 + |x|^2 ).
\end{eqnarray*}

(ii) From calculations for \eqref{e:nuf}, one has
	\begin{eqnarray*}
\nabla_u f_{\delta}(x)
&=&\int_0^{\infty} e^{-\lambda t} (  \lambda  \E [ f_{\delta} (X_t^x)  \mathcal{I}_{u}^x(t) ]  -  \E [  h_{\delta}(X_t^x) \mathcal{I}_{u}^x(t) ] ) \dif t.
	\end{eqnarray*}
	
With similar calculations for the proof of \eqref{e:SE2}, one has
\begin{eqnarray*}
|\nabla_{u} f_{\delta}(x)|
&\leq&  C \| h_{\delta} \|_{\infty}  (1+ |x|^2) |u|
\ \leq \ C \| h \|_{\infty}(1+|x|^2) |u|.
\end{eqnarray*}

	Since the operator $\nabla$ is closed, by using the dominated convergence theorem, one has
	\begin{eqnarray*}
	         \lim_{\delta \to 0} \nabla_u f_{\delta}(x)
	         &=& \lim_{\delta \to 0}  \int_0^{\infty} e^{-\lambda t} (  \lambda  \E [ f_{\delta} (X_t^x)  \mathcal{I}_{u}^x(t) ]  -  \E [  h_{\delta}(X_t^x) \mathcal{I}_{u}^x(t) ] ) \dif t \\
	         &=&  \int_0^{\infty} e^{-\lambda t} (  \lambda  \E [ f (X_t^x)  \mathcal{I}_{u}^x(t) ]  -  \E [  h(X_t^x) \mathcal{I}_{u}^x(t) ] ) \dif t \\
	         &=& \nabla_u f(x).
	\end{eqnarray*}

	 Letting $\delta \to 0$, we know
     \begin{eqnarray*}
|\nabla_{u} f(x)|
&\leq&  C \| h \|_{\infty}  (1+ |x|^2) |u|.
	\end{eqnarray*}
The proof is complete.
\end{proof}

\end{appendix}

\section*{Acknowledgements}
L. Xu is supported in part by NSFC grant (No. 12071499), Macao S.A.R grant FDCT  0090/2019/A2 and University of Macau grant  MYRG2018-00133-FST.
G. Pang is supported in part by  the US National Science Foundation grants DMS-1715875 and DMS-2108683, and Army Research Office grant W911NF-17-1-0019.

\end{document}